\documentclass[11pt,letterpaper]{amsart}
\usepackage{amsmath,amstext,amsthm,amscd,typearea}
\usepackage{amssymb}
\usepackage{a4wide}
\usepackage[mathscr]{eucal}
\usepackage{mathrsfs}
\usepackage{typearea}
\usepackage{charter}
\usepackage{pdfsync}
\usepackage{url}

\usepackage{xcolor}

\usepackage[a4paper,width=15.2cm,top=4cm,bottom=4cm]{geometry}

\numberwithin{equation}{section}

\newtheorem{theorem}{Theorem}[section]

\newtheorem{proposition}[theorem]{Proposition}
\newtheorem{corollary}[theorem]{Corollary}
\newtheorem{lemma}[theorem]{Lemma}

\newcommand{\cali}[1]{\mathscr{#1}}

\newcommand{\Leb}{\mathop{\mathrm{Leb}}\nolimits}

\newcommand{\dist}{\mathop{\mathrm{dist}}\nolimits}

\newcommand{\vol}{\mathop{\mathrm{vol}}}

\newcommand{\ddc}{dd^c}

\def\d{\operatorname{d}}

\newcommand{\PSH}{{\rm PSH}}

\newcommand{\codim}{{\rm codim\ \!}}

\newcommand{\Cc}{\cali{C}}

\newcommand{\capK}{\text{cap}}

\newcommand{\B}{\mathbb{B}}
\newcommand{\C}{\mathbb{C}}

\newcommand{\N}{\mathbb{N}}
\newcommand{\Z}{\mathbb{Z}}
\newcommand{\R}{\mathbb{R}}
\renewcommand\P{\mathbb{P}}

\title{\bf Log continuity of solutions of complex Monge-Amp\`ere equations}

\providecommand{\subject}[1]{\textbf{\textit{Mathematics Subject Classification 2020:}} #1}

\author{Hoang-Son Do and Duc-Viet Vu}
\newcommand{\Addresses}{{
		\bigskip
		\footnotesize
		\textsc{Duc-Viet Vu, University of Cologne, Division of Mathematics, Department of Mathematics and Computer Science, Weyertal 86-90, 50931, K\"oln,  Germany}
		\noindent
		\par\nopagebreak
		\noindent
		\textit{E-mail address}: \texttt{dvu@uni-koeln.de}	
		
	\bigskip
\footnotesize
\textsc{Hoang-Son Do, Vietnam Academy of Science and Technology, Institute of Mathematics, 18 Hoang Quoc Viet road, Cau Giay, Hanoi, Vietnam}
\noindent
\par\nopagebreak
\noindent
\textit{E-mail address}: \texttt{dhson@math.ac.vn}	}}
\date{\today}
\begin{document}
\maketitle
\begin{abstract}  Let $X$ be a compact K\"ahler manifold whose anticanonical cohomology class is semipositive. Let $L$ be a big  and semi-ample line bundle on $X$ and $\alpha$ be the Chern class of $L$. We give a sufficient condition ensuring that the solution of the complex Monge-Amp\`ere equations in $\alpha$ with $L^p$ right-hand side ($p>1$) is $\log^M$-continuous for every constant $M>0$. As an application, we show that every singular Ricci-flat metric in a semi-ample integral class in a projective Calabi-Yau surface $X$ is globally $\log^M$-continuous with respect to a smooth metric on $X$. 
\end{abstract}
\noindent
%\keywords {Monge-Amp\`ere equation}, {Calabi-Yau manifold}, {Log continuity}, {Fano manifold}.
\\

\noindent
\subject{32U15}, {32Q15}, {53C55}.

%\tableofcontents

\section{Introduction}

Let $(X, \omega)$ be a compact K\"ahler manifold. A cohomology $(1,1)$-class $\alpha \in H^{1,1}(X,\R)$ is said to be semi-positive if $\alpha$ contains a  semi-positive smooth form.   Let $\theta$ be a  smooth closed $(1,1)$-form in a big and semi-positive cohomology class. We consider the following complex Monge-Amp\`ere equation 
\begin{align}\label{eq-MA}
(\ddc u+ \theta)^n= f\omega^n,\quad\sup_X u=0,
\end{align}
where $f \in L^p$ ($p>1$) is a nonnegative function so that $\int_X f \omega^n =\int_X \theta^n$. The regularity of solutions of (\ref{eq-MA}) is well-known if $\theta$ is K\"ahler thanks to pioneering works by Yau \cite{Yau1978} and Ko{\l}odziej \cite{Kolodziej_Acta}, and many subsequent papers. We refer to \cite{DemaillyHiep_etal,DKC_Holder-Sobolev,DinhVietanhMongeampere,KC-holder2018,Kolodziej08holder,Lu-To-Phung,NgocCuong-Holder-calc.var,NgocCuong-Holder2020,Tosatti-Weinkove,Vu_MA,Vu_MA_holder} and references therein for details on H\"older continuity of solutions when $\theta$ is K\"ahler. 

The focus of our work is the case where $\theta$ belongs to  a semi-positive  and  big cohomology class. In this general setting, 
it is  well-known by \cite{BEGZ} that the solution $u$ is smooth outside the non-K\"ahler locus of the cohomology class of $\theta$. By \cite{EGZ,Coman-Guedj-Zeriahi} or \cite{Dinew_Zhang_stability}, we know that the equation (\ref{eq-MA}) admits a unique continuous solution $u$ on $X$ if the cohomology class of $\theta$ is integral (see  \cite{GGZ} for more information). %It was actually conjectured that it is the case for every big semi-positive form $\theta$ which is not necessarily integral. This question is still widely open. The aim of this paper is to quantify this continuity of solutions. The methods in  \cite{EGZ,Coman-Guedj-Zeriahi} or \cite{Dinew_Zhang_stability} offer no clue how to improve the continuity of $u$ in a quantitative way. 
The aim of this paper is to quantify this continuity property of solutions. The methods in  \cite{EGZ,Coman-Guedj-Zeriahi} or \cite{Dinew_Zhang_stability} seem to be only qualitative. 
To state our results, we need to introduce some notions. 

%We say that a cohomology class $\alpha \in H^{1,1}(X,\R)$ is \emph{semi-positive} if $\alpha$ contains a closed semipositive smooth form. 

Let $M>0$ be a constant. We fix a smooth Riemannian metric $\dist(\cdot, \cdot)$ on $X$. A function $u$ on $X$ is said to be $\log^M$-continuous if there exists a constant $C_M>0$ such that
$$|u(x)- u(y)| \le \frac{C_M}{|\log \dist(x, y)|^M},$$
for every $x, y \in X$. Let $K_X$ be the canonical line bundle of $X$. Recall that $X$ is Calabi-Yau  if $c_1(K_X)=0$, and  $X$ is Fano if $-c_1(K_X)$ is ample. A line bundle $L$ on $X$ is said to be semi-ample if $L^k$ is base-point free for some large enough integer $k \ge 1$. It is a well-known fact  (see \cite[Section 2]{Tosatti-noncollapse} for a summary) that $L$ is semi-ample if $X$ is a projective Calabi-Yau manifold and $L$ is big and nef.

Let $L$ be a big and semi-ample line bundle on $X$ (hence $X$ is forced to be projective by Moishezon's theorem). 
Let $d_k:= \dim H^0(X,L^k)$ and $\{s_1, \ldots, s_{d_k}\}$ be a basis of $H^0(X, L^k)$. We define $\Phi_k: X \to \C\P^{d_k-1}$ by putting 
$$\Phi_k(x):= [s_1(x): \cdots: s_{d_k}(x)].$$
Observe that $\Phi_k$ is a well-defined map outside $B(kL):= \cap_{s \in H^0(X, L^k)}\{s=0\}$. 
%We recall $\ddc:= \frac{i}{\pi} \partial \bar \partial$.
Since $L$ is semi-ample and big, % there is $k'>0$ sufficiently large so that $B(k' L) = \varnothing$. Hence  $\Phi_{k'}: X \to \C\P^{d_k}$ is a holomorphic map. Since $L$ is big, we can find $k''>0$ so that $\Phi_{k''}$ is of maximal rank. Let $k_L:= k' k''$. 
there exists an integer $k_L>0$ such that $\Phi_{k_L}$ is bihomorphic outside the non-K\"ahler locus $N$ of $c_1(L)$ (see \cite[Theorem A]{Boucksom-Cacciola-Lopez}) and  $\Phi_{k_L}$ is an algebraic fibre space, \emph{i.e}, the fibers of $\Phi_{k_L}$ are connected  $X':= \Phi_{k_L}(X)$ is a normal variety  (see \cite[Theorem 2.1.27]{Lazarsfeld}).
Here is our main result in this work giving a partial answer to the above question. 

\begin{theorem}\label{the-log-continuity} Assume that $-c_1(K_X)$ is semi-positive  and  $X'$ has only isolated singularities. Let $\theta \in c_1(L)$ be a smooth form. Then the unique solution $u$ of (\ref{eq-MA}) is $\log^M$-continuous for every constant $M>0$. 
\end{theorem}

Our proof shows actually that Theorem \ref{the-log-continuity} is still true if $-c_1(K_X)$ contains a closed positive $(1,1)$-current of bounded potentials (see Subsection \ref{subsec-regu}).  Since $X'$ is normal, if $n=2$, then $X'$ always has isolated singularities. 
Theorem \ref{the-log-continuity} is probably the first known quantitative (global) regularity for solutions of complex Monge-Amp\`ere equations in a semi-positive class. We would like to notice that  it was proved in \cite{Guo-Phong-Tong-Wang} that the solution of the equation $(\ddc u+ \omega)^n= e^F \omega^n$ for $e^F \in L^1 (\log L)^p$ is  $\log^M$-continuous for $M:= \min \{\frac{p-n}{n}, \frac{p}{n+1}\}$; see also \cite{GGZ-logcontiu} for a recent development. As far as we can see, the method in \cite{Guo-Phong-Tong-Wang} or \cite{GGZ-logcontiu}  uses crucially the fact that $\omega$ is K\"ahler and it is not clear if this can be extended to semi-positive classes to obtain a $\log^M$-continuity for solutions of (\ref{eq-MA}). %As commented there, the $\log$ continuity is very rare encountered in the literature of K\"ahler geometry. Our Theorem \ref{the-log-continuity} hence offer another setting where one can observe the $\log$ continuity of complex Monge-Amp\`ere equations. 

Assume that $X, L,\omega$ are as in the statement of Theorem \ref{the-log-continuity}. Hence the non-K\"ahler locus $N$ of $c_1(L)$  is a proper analytic subset in $X$; see \cite{Boucksom_anal-ENS}. Let $F$ be a  smooth function on $X$ such that $\int_X e^F \omega^n = \int_X (c_1(L))^n$ and denote by $\omega_F$  the (singular) positive  $(1,1)$-form on $X$ such that $\omega_F^n= e^F \omega^n$. Recall that $\omega_F$ is a genuine K\"ahler metric on $X \backslash N$.  

\begin{corollary}\label{cor-metricsemiample} Let $X, L,\omega, N$ and $F$ be as above and the hypothesis in Theorem \ref{the-log-continuity} holds.  Then for every constant $M>0$, there exists a constant $C_M>0$ such that 
\begin{align}\label{ine-revisdedomegalog}
d_{\omega_F}(x,y) \le C_M |\log \dist(x,y)|^{-M},
\end{align}
for every $x,y \in X \backslash N$, where $d_{\omega_F}$ is the distance induced by $\omega_F$ on $X \backslash N$.
\end{corollary}

In the case where $\theta$ is in a K\"ahler class, one has better estimates; see \cite{GGZ-logcontiu,YangLi,Vu-log-diameter} for details. We are not aware of any previous result similar to Corollary \ref{cor-metricsemiample} for merely semi-ample and big classes. The inequality (\ref{ine-revisdedomegalog}) implies, in particular, that $d_{\omega_F}$ is bounded. This fact can be also deduced from a much more general diamter estimate established recently in \cite{Guo-Phong-Song-Sturm,Guo-Phong-Song-Sturm2,GPSS_bodieukien} (see also  \cite{GuedjTo-diameter,Vu-diameter}).  As an immediate consequence of Corollary \ref{cor-metricsemiample}, we get the following. 

\begin{corollary} \label{cor-chinhquymetric} Let $X$ be a compact Calabi-Yau surface. Let  $L$ be a semi-ample and big line bundle on $X$ and $\omega_0$ is a (singular) K\"ahler Ricc-flat metric in $c_1(L)$. Then $\omega_0$ has a $\log^M$-continuous potential. Moreover, if $d_{\omega_0}$ denotes the distance induced by $\omega_0$ on $X \backslash N$ (where $N$ is non-K\"ahler locus of $c_1(L)$) then for every constant $M>0$ there is a constant $C_M>0$ so that   
$$d_{\omega_0}(x,y) \le C_M |\log \dist(x,y)|^{-M},$$
for every $x,y \in X\backslash N$.  
\end{corollary}

%One can apply Corollary \ref{cor-chinhquymetric} to the case where $X$ is a Calabi-Yau surface for which the hypothesis of Corollary \ref{cor-chinhquymetric} is automatically satisfied. In this case
We recall that the existence of singular Ricci-flat metric $\omega_0$ of bounded potentials was proved  in \cite{EGZ}, and the non-quantitative continuity of $\omega_0$ (for any dimension) was obtained in \cite{Coman-Guedj-Zeriahi} and \cite{Dinew_Zhang_stability}. % in which it was shown also that $\omega_0$ is of continuous potentials provided that an approximation condition on $c_1(L)$ is satisfied. 

We now explain main ideas in the proof of Theorem \ref{the-log-continuity}. We will need to approximate our smooth solution $u$ by smooth quasi-psh function $(u_\epsilon)_\epsilon$. Using \cite{Demailly_appro_chernconnec} or \cite[Theorem 4.12]{Demailly_analyticmethod} (analytic approximation for general closed positive $(1,1)$-currents), one obtains $(\theta+ \epsilon \omega)$-psh functions $u_\epsilon$ so that $u_\epsilon$ converges to $u$ in a quantitative way in $L^1$. However $\|\nabla u_\epsilon\|_{L^\infty}$   grows like $e^{1/ \epsilon}$. The fact that $u_\epsilon$ is only $(\theta+\epsilon \omega)$-psh and a bad control on $\|\nabla u_\epsilon\|_{L^\infty}$ is not usable in our approach. For this reason, we have to restrict ourselves to the line bundle setting for which a more precise approximation procedure is available. Precisely we will need a modified version of Demailly's analytic approximation of  singular (not necessarily K\"ahler) Hermitian metrics for a line bundle (Theorem \ref{th-analytic-approx}), this is the place where we need to use the hypothesis that $X'$ has only isolated singularities. This regularisation together with Ko{\l}odziej's capacity technique will give us a weak Log continuity property for $u$ (see Lemma \ref{le-logcontinuity}). Our second ingredient (Sections \ref{sec-log} and \ref{sec-log2}) is to say that a function satisfying this weak Log continuity property is indeed Log continuous as desired.

The paper is organized as follows. In Sections \ref{sec-log} and \ref{sec-log2}, we present important facts about log continuity of functions. In Section \ref{sec-capa}, we recall some facts about H\"older continuous measures. In Section \ref{sec.appr}, we present a 
 modified version of Demailly's analytic approximation.
The rest of the paper is devoted to the proof of main results. \\

\noindent
\textbf{Acknowledgement.} We thank George Marinescu for fruitful discussions. We are grateful to Jian Song for pointing out an error in a previous version of this paper.   The research of D.-V. Vu is partially funded by the Deutsche Forschungsgemeinschaft (DFG, German Research Foundation)-Projektnummer 500055552 and by the ANR-DFG grant QuaSiDy, grant no ANR-21-CE40-0016.
The research of H.-S. Do  is  funded by International Centre for Research and Postgraduate Training in Mathematics (ICRTM, Vietnam) under grant number ICRTM01 \_ 2023.01.

\section{Log continuity of pseudometrics}\label{sec-log}

Let $Z$ be a topological space and $d: Z^2 \to \R_{\ge 0}$ be a function. Let $B \ge 1$ be a constant. We say that $d$ is \emph{a $B$-pseudometric on $Z$} if the following holds:

(i) $d(x, x)=0$  for every $x\in Z$;

(i) $d$ is symmetric, i.e, $d(x,y)= d(y,x)$ for every $x, y\in Z$;

(ii) $d$ is continuous on $Z^2$;

(iii) for every $x_1,\ldots, x_m \in Z$, one has 
$$d(x_1, x_{m}) \le B \sum_{j=1}^{m-1} d(x_j,x_{j+1}).$$

%Let $U \subset \R^m$ be a bounded open subset. Let $u: U \to \R$ be a function and let $M>0$ be a constant. We say that $u$ is $\log^M$-continuous if there is a constant $C_M>0$ such that
%$$|u(x)- u(y)| \le C_M |\log |x-y||^{-M}$$ 
%for every $x,y \in U$. 

\begin{lemma}\label{lemstep2}
	Let $U\subset\R^m$ be a bounded convex domain ($m\geq 2$). Let $B \ge 1$ be a constant.  Let $d:U\times U\rightarrow [0, \infty)$ be a 
	$B$-pseudometric satisfying the following condition:
there exist constants $\alpha>0$, $D>1$ and $C_0>0$ such that
\begin{equation}\label{eq0lemstep2}
		d(x, y)\leq \dfrac{C_0}{|\log |x-y||^{\alpha}},
\end{equation}
		for every $x, y\in U$ with $|x-y|^{D} \le \min\{\dist(x, \partial U), \dist(y, \partial U)\}$, where
		$$\dist (w, \partial U)=\inf\{|w-\xi|: \xi\in\partial U\}.$$
	Then, there exists a  constant $C>0$ depending only on $B, C_0, \alpha, D$ and $U$  such that
	$$d(x, y)\leq  \dfrac{C}{|\log |x-y||^{\alpha}},$$
	for every $x, y\in U$.
\end{lemma}
\begin{proof}
	Without loss of generality, we can assume that $diam (U)\leq 1$. In particular, $|x-y|^{D}\leq |x-y|$ for every $x, y\in U$. 
	
		Fix $a\in U$ and denote
		$r=\dist (a, \partial U)$. The desired assertion is clear if we have either   $|x-y| \ge r/2$ or   
		$$\min\{\dist(x, \partial U), \dist(y, \partial U)\}
\geq r/2\geq |x-y|$$
 (by \eqref{eq0lemstep2}). 
 %By the condition \eqref{eq0lemstep2}, we have	$$d(x, y)\leq  \dfrac{C_0}{|\log |x-y||^{\alpha}}.$$
%{\bf Case 2:} 
Consider now the case where  
$$ |x-y|\leq r/2\quad\mbox{and}\quad\min\{\dist(x, \partial U), \dist(y, \partial U)\}
	< r/2.$$
Thus, we have $\max\{|x-a|, |y-a|\}>|x-y|$. Without loss of generality, we can assume that
$|x-a|>|x-y|:=\delta>0$. Set 
$$x_0=\dfrac{(|x-a|-\delta)x}{|x-a|}+\dfrac{\delta a}{|x-a|}.$$
In other words, $x_0$ is a point in $[x, a]$ satisfying $|x-x_0|=\delta$. Since $U$ is convex, we have
\begin{equation}
	\dist(x_0, \partial U)\geq \dfrac{(|x-a|-\delta)\dist(x, \partial U)}{|x-a|}+\dfrac{\delta \dist(a, \partial U)}{|x-a|}
	\geq r \delta.
\end{equation}
For every $k\in\Z^+$, we denote by $x_k$ the point in $[x, x_0]$ satisfying $|x-x_k|=\delta^{D^k}$.
Then, we have
	$$\dist(x_k, \partial U)\geq \dfrac{\dist(x_0, \partial U)|x_k-x|}{|x-x_0|}\geq r \delta^{D^k}\quad
	\mbox{and}\quad |x_k-x_{k-1}|\leq \delta^{D^{k-1}}.$$
Put $M=\left[\dfrac{1}{r}\right]+1$, where $[\cdot]$ is the greatest integer function. For every 
$l=0,..., M$ and $k\in\Z^+$, we denote
$$x_{k, l}=x_{k-1}+\dfrac{l (x_k-x_{k-1})}{M}.$$
Then $|x_{k, l}-x_{k, l+1}|\leq \dfrac{\delta^{D^{k-1}}}{M}$. Moreover, since $\dist(., \partial U)$ is a concave function on $U$,
 we have
 $$\dist(x_{k, l}, \partial U)\geq\min\{\dist(x_k, \partial U), \dist(x_{k-1}, \partial U)\}\geq  r \delta^{D^k}
 \geq\dfrac{\delta^{D^k}}{M}.$$
 Therefore, by the condition \eqref{eq0lemstep2}, we get
 $$d(x_{k,l}, x_{k, l+1})\leq \dfrac{C_0}{|\log |x_{k, l}-x_{k, l+1}||^{\alpha}}\leq
  \dfrac{C_0}{D^{(k-1)\alpha}|\log\delta|^{\alpha}}.$$
  Thus, we have
  $$d(x_k, x_{k-1})\leq B \sum_{l=0}^{M-1}d(x_{k,l}, x_{k, l+1})\leq
   \dfrac{B C_0M}{D^{(k-1)\alpha}|\log\delta|^{\alpha}}.$$
Hence
$$d(x_k, x_0)\leq B \sum_{j=1}^{k}d(x_j, x_{j-1})
\leq  \dfrac{B^2 C_0M}{|\log\delta|^{\alpha}}\sum_{j=1}^kD^{-(j-1)\alpha}\leq  \dfrac{C_1}{|\log\delta|^{\alpha}},$$
where $C_1=\dfrac{B^2 C_0 M\, D^{-\alpha}}{1-D^{-\alpha}}=\dfrac{B^2C_0M}{D^{\alpha}-1}$.

Since $d$ is continuous on $U\times U$, one gets
\begin{equation}\label{eq2step2}
	d(x, x_0)=\lim_{k\to\infty}d(x_k, x_0)\leq \dfrac{C_1}{|\log\delta|^{\alpha}}.
\end{equation}
Since $|y-x_0|\leq |x-y|+|x-x_0|\leq 2\delta$, by using the same argument as above, we also have
\begin{equation}\label{eq3step2}
	d(y, x_0)\leq \dfrac{C_2}{|\log\delta|^{\alpha}},
\end{equation}
where $C_2>0$ depends only on $B, C_0, M, D$ and $\alpha$.

Combining \eqref{eq2step2} and \eqref{eq3step2}, we get
$$d(x, y)\leq B(d(x, x_0)+d(y, x_0))\leq \dfrac{B(C_1+C_2)}{|\log\delta|^{\alpha}}
=\dfrac{B(C_1+C_2)}{|\log |x-y||^{\alpha}}.$$
%{\bf Case 3:}  $|x-y|> r/2$.
%Note that$|x-y|\leq diam(U)\leq 1 .$
 %For every 
%$l=0,..., 2M$, we denote
%$$w_{l}=x+\dfrac{l (y-x)}{2M}.$$
%We have $|w_l-w_{l-1}|\leq r/2$ for every $l=1, 2, ..., 2M$. By Case 1 and Case 2, we get
%$$d(w_l, w_{l-1})\leq \dfrac{C_3}{|\log |w_l-w_{l-1}||^{\alpha}}\leq \dfrac{C_3}{|\log |x-y||^{\alpha}},$$
%for every $l=1, 2, ..., 2M$, where $C_3>0$ depends only on $C_0, \alpha, D$ and $U$. Hence
%$$d(x, y)\leq \sum_{l=1}^{2M}d(w_l, w_{l-1})\leq \dfrac{2MC_3}{|\log |x-y||^{\alpha}}.$$
The proof is completed.
\end{proof}

\begin{lemma}\label{lem2step2}
	Let $U\subset\R^m$ be a bounded convex domain ($m\geq 2$). Let $B \ge 1$ be a constant. Assume $d:U\times U\rightarrow [0, \infty)$ is a 
	$B$-pseudometric satisfying the following condition:
	there exist constants $\alpha>1$ and $C_0>0$ such that
	\begin{equation}\label{eq0lem2step2}
		d(x, y)\leq \dfrac{C_0}{|\log |x-y||^{\alpha}},
	\end{equation}
	for every $x, y\in U$ with $|x-y| \le \min\{\dist(x, \partial U), \dist(y, \partial U)\}$, where
	$$\dist (w, \partial U)=\inf\{|w-\xi|: \xi\in\partial U\}.$$
	Then, there exists a  constant $C>0$ depending only on $B, C_0, \alpha, D$ and $U$  such that
	$$d(x, y)\leq  \dfrac{C}{|\log |x-y||^{\alpha-1}},$$
	for every $x, y\in U$.
\end{lemma}
\begin{proof} 
	 We will use the same method as in the proof of Lemma \ref{lemstep2}.
	 Without loss of generality, we can assume that there exists $a\in U$ such that
	 $r=\dist (a, \partial U)\geq 1$.
	We only need to consider the case  where $ |x-y|\leq r/2\quad\mbox{and}\quad\min\{\dist(x, \partial U), \dist(y, \partial U)\}	< r/2$. 
		In this case, we have $\max\{|x-a|, |y-a|\}>|x-y|$. We can  assume that
	$|x-a|>|x-y|:=\delta>0$. Set 
	$$x_0=\dfrac{(|x-a|-\delta)x}{|x-a|}+\dfrac{\delta a}{|x-a|}.$$
	In other words, $x_0$ is a point in $[x, a]$ satisfying $|x-x_0|=\delta$. Since $U$ is convex, we have
	\begin{equation}
		\dist(x_0, \partial U)\geq \dfrac{(|x-a|-\delta)\dist(x, \partial U)}{|x-a|}+\dfrac{\delta \dist(a, \partial U)}{|x-a|}
		\geq r \delta.
	\end{equation}
	For every $k\in\Z^+$, we denote by $x_k$ the point in $[x, x_0]$ satisfying $|x-x_k|=2^{-k}\delta$.
	Then, we have
	$$\dist(x_k, \partial U)\geq \dfrac{\dist(x_0, \partial U)|x_k-x|}{|x-x_0|}\geq r 2^{-k}\delta\quad
	\mbox{and}\quad |x_k-x_{k-1}|\leq 2^{-k}\delta.$$
By the condition \eqref{eq0lem2step2}, we have
	$$d(x_{k}, x_{k-1})\leq \dfrac{C_0}{|\log |x_{k}-x_{k-1}||^{\alpha}}\leq
	\dfrac{C_0}{\left(|\log\delta|+k\log 2\right)^{\alpha}},$$
		for every $k\in\Z^+$.
	Hence
	$$d(x_k, x_0)\leq B \sum_{j=1}^{k}d(x_j, x_{j-1})
	\leq  \sum_{j=1}^k	\dfrac{B C_0}{\left(|\log\delta|+j\log 2\right)^{\alpha}}
	\leq\frac{B C_0}{\log 2}\int_{\log|\delta|}^{\infty}\frac{dt}{t^{\alpha}}	\leq  \dfrac{C_1}{|\log\delta|^{\alpha-1}},$$
	where $C_1=\frac{B C_0}{(\alpha-1)\log 2}$.
	
	Since $d$ is continuous on $U\times U$, one has
	\begin{equation}\label{eq2lem2step2}
		d(x, x_0)=\lim_{k\to\infty}d(x_k, x_0)\leq \dfrac{C_1}{|\log\delta|^{\alpha-1}}.
	\end{equation}
	Since $|y-x_0|\leq |x-y|+|x-x_0|\leq 2\delta$, by using the same argument as above, we also have
	\begin{equation}\label{eq3lem2step2}
		d(y, x_0)\leq \dfrac{C_2}{|\log\delta|^{\alpha-1}},
	\end{equation}
	where $C_2>0$ depends only on $B, C_0$ and $\alpha$.
	Combining \eqref{eq2lem2step2} and \eqref{eq3lem2step2}, we get
	$$d(x, y)\leq B(d(x, x_0)+d(y, x_0))\leq \dfrac{B(C_1+C_2)}{|\log\delta|^{\alpha-1}}
	=\dfrac{B(C_1+C_2)}{|\log |x-y||^{\alpha-1}}.$$

	The proof is completed.
\end{proof}

\begin{proposition} \label{pro-chuyendudistNdenlogcon}
	Let $N_1, N_2..., N_p$ be  affine subspaces of $\R^m$ such that $\codim (N_j)\geq 2$ for
	every $j=1,...,p$. Denote $N=\cup_{j=1}^p N_j$. Let $B \ge 1$, $\alpha>0$, $D\geq 1$ and $C_0>0$ be constants.   Let $d$ be  a $B$-pseudometric on $\B^m\setminus N$	satisfying one of the following conditions
	\begin{itemize}
		\item [(i)]    $D>1$ and 
		\begin{equation}\label{eq1prostep2}
		d(x, y)\leq \dfrac{C_0}{|\log |x-y||^{\alpha}},
		\end{equation}
		for every $x, y\in \B^{m}\setminus N$ with $|x-y|^D \le \min\{\dist(x, N), \dist(y, N)\}$.
		\item[(ii)] $D=1$ and
			\begin{equation}\label{eq1.1prostep2}
			d(x, y)\leq \dfrac{C_0}{|\log |x-y||^{\alpha+1}},
			\end{equation}
			for every $x, y\in \B^{m}\setminus N$ with $|x-y| \le \min\{\dist(x, N), \dist(y, N)\}$.
	\end{itemize}
Then, there exists $C>0$ depending only on $B, C_0, \alpha, D, N$ and $m$ such that
$$d(x, y)\leq \dfrac{C}{|\log |x-y||^{\alpha}},$$
for every $x, y\in\B^{m}\setminus N$.
\end{proposition}

\begin{proof}
	We will give the proof for the first case where $(i)$ is satisfied. The  second case is similar
	(use Lemma \ref{lem2step2} in place of Lemma \ref{lemstep2}). % and we leave it to the reader.
	%By the condition \ref{eq1prostep2} and by the triangle inequality, we
	Recall that  $d$ is a continuous function on $(\B^m\setminus N)\times (\B^m\setminus N)$.
	Let $H_j$ be a hyperplane containing $N_j$ for  $j=1,..., p$, and denote $H=\cup_{j=1}^pH_j$. Observe that the connected components of $\B^{m}\setminus H$ are bounded convex subsets of $\R^m$. Moreover, if $U$ is a connected
	component of $\B^{m}\setminus H$ then by \eqref{eq1prostep2}, $d$ satisfies the condition \eqref{eq0lemstep2}
	 in Lemma \ref{lemstep2}.

	Let $x, y\in \B^{m}\setminus N$. We distinguish into three cases.
	
	{\bf Case 1:} there exists a connected component $U$ of $\B^{m}\setminus H$ such that
	 $x, y\in\overline{U}\setminus N$.\\
	In this case, by Lemma \ref{lemstep2} and by the continuity of $d$, we have
	$$d(x, y)\leq \dfrac{C_U}{|\log |x-y||^{\alpha}},$$
	 where $C_U>0$ is a constant depending only on $B, C_0, \alpha, D$ and $U$.
	 
	 {\bf Case 2:} $[x, y]\cap H\neq\emptyset$ but $[x, y]\cap N=\emptyset$.\\
	 In this case, there exist connected components $U_1, U_2, ..., U_k$  of $\B^{m}\setminus H$ 
	  and $x_0, x_1, x_2,...,x_{k}\in [x, y]$ such that $x_0=x\in\overline{U_1}$,
	 $x_k=y\in\overline{U_k}$ and
	 $x_j\in \partial U_j\cap\partial U_{j+1}$ for every $j=1,...,k-1$. Using the result in Case 1, we have
	 \begin{align*}
	 	d(x, y)\leq B \sum_{j=1}^{k}d(x_j, x_{j-1})&\leq \sum_{j=1}^k\dfrac{B C_{U_j}}{|\log |x_j-x_{j-1}||^{\alpha}}\\
	 	&\leq \dfrac{k B C_1}{|\log |x-y||^{\alpha}}\\ 
	 	&\leq \dfrac{(p+1)B C_1}{|\log |x-y||^{\alpha}},
	 \end{align*}
 where $C_1=\sup\{C_U: U $ is a connected component of $ \B^m\setminus N \}$.
 
 {\bf Case 3:} $[x, y]\cap N\neq\emptyset$.\\
 Denote $f(t)=tx+(1-t)y$, $0\leq t\leq 1$. Then, there exist $0<k\leq p$ and $0<t_1<t_2<...<t_k<1$ such that
 $$[x, y]\cap N=\{f(t_j): j=1,...,k\}.$$
 By Lemma \ref{le-chonduongcong} below, for every $j=1,...,k$ and for every $0<\epsilon\ll 1$,
  there exists
 a piecewise linear curve $l=a_0a_1...a_{4^p}$ with $a_0=f(t_j+\epsilon)$ and $a_{4^p}=f(t_j-\epsilon)$ such that
 $l$ does not intersect $N$ and 
 $$L(l)\leq C_2|f(t_j+\epsilon)-f(t_j-\epsilon)|=2 C_2\epsilon|x-y|,$$
 where $C_2\geq 1$ is a constant depending only on $p$. Therefore, by the result in
 Case 2, we have
 \begin{equation}\label{eq2prostep2}
 	 d(f(t_j+\epsilon), f(t_j-\epsilon)) =O(\epsilon).
 \end{equation}
Denote $t_0=0$ and $t_{k+1}=1$. By Case 2 and by \eqref{eq2prostep2}, we have

 \begin{align*}
 	d(x, y)&=\lim_{\epsilon\to 0+} d(f(t_0+\epsilon), f(t_{k+1}-\epsilon))\\
 	&\leq \limsup_{\epsilon\to 0+}B\sum_{j=0}^kd(f(t_j+\epsilon), f(t_{j+1}-\epsilon))
 	+\limsup_{\epsilon\to 0+}B\sum_{j=1}^kd(f(t_j+\epsilon), f(t_{j}-\epsilon))\\
 	&\leq \limsup_{\epsilon\to 0+}\sum_{j=0}^k \dfrac{B^2C_1(p+1)}{|\log |f(t_j+\epsilon)- f(t_{j+1}-\epsilon)||^{\alpha}}\\
 	&\leq \dfrac{B^2C_1(p+1)^2}{|\log |x-y||^{\alpha}}.
 \end{align*}
 The proof is completed.
\end{proof}

The following lemma plays also an important role in our proof later. 

\begin{lemma} \label{le-chonduongcong} Let $N_1, N_2,..., N_k$ be affine subspaces of $\R^m$ such that
	$\codim (N_j)\geq 2$ for every $j=1,...,k$. Denote $N=\cup_{j=1}^kN_j$.
	 Then,  there is a constant $C\geq 1$ depending only on $k$ (and $m$)  satisfying the following property:
	  for every $x, y\in \R^m\setminus N$, there is a polygonal chain $l=a_0a_1...a_{4^k}$ with
	   $a_0=x$ and $a_{4^k}=y$ such that
	 $$C \dist(\xi, N) \ge  \min \{\dist(x,N), \dist(y,N)\},$$ 
	 for every $\xi\in \cup_{s=0}^{4^k-1}[a_s, a_{s+1}]$, and
	 	$$L(l) \le C |x-y|,$$
	 	 where $L(l)=|a_0-a_1|+|a_1-a_2|+...+|a_{4^k-1}-a_{4^k}|$ is the length of $l$. 
\end{lemma}

In order to prove Lemma \ref{le-chonduongcong}, we need the following elementary lemma:

\begin{lemma}\label{le-chonduongcongA}
	Let $N$ be an affine subspace of $\R^m$ with $\codim N\geq 2$. Let $r\geq 1$ be a constant. Then, for every $x, y\in\R^m\setminus N$, there
	exists $w\in\R^m\setminus N$ such that
	$$|x-w|+|w-y|\leq 3|x-y|,$$
	and
	$$2r\dist (\xi, N)\geq\min\{\dist(x, N), \dist (y, N)\}\geq r\dist(\xi, [x, y]),$$
	for every $\xi\in [x, w]\cup [w, y]$.
\end{lemma}

\proof   Observe that the function $\dist (\cdot, N)$ is convex on $\R^m$. Indeed, for every $a, b\in\R^m$,
there exist $a_0, b_0\in\N$ such that $|a-a_0|=\dist(a, N)$ and $|b-b_0|=\dist (b, N)$. Hence, if 
$\eta= \alpha a + (1- \alpha) b$ for some $\alpha\in [0, 1]$ then
\begin{align*}
\alpha\dist (a, N)+(1-\alpha)\dist (b, N)&=\alpha|a-a_0|+(1-\alpha)|b-b_0|\\
&\geq |\alpha(a-a_0)+(1-\alpha)(b-b_0)|\\
&=|\eta-(\alpha a_0
+(1-\alpha)b_0)|\geq \dist (\eta, N).
\end{align*}

Let $R:= \min \{\dist(x, N), \dist(y,N)\}$. If  $\dist(\eta, N) \ge R/(2r)$ for every $\eta \in [x,y]$, then $w:=x$ satisfies the desired property. Assume,  from now on, that there is a point $\eta \in [x,y]$ such that $\dist(\eta, N) \le R/(2r) \le R/2$. We deduce that 
\begin{align}\label{ine-sosanhxyR}
|x-y|=|x-\eta|+ |\eta- y| \ge \dist(x,N)- \dist(\eta, N)+ \dist(y,N)- \dist(\eta, N) \ge R.
\end{align} 
We distinguish into three cases\\

{\bf Case 1:} Either $[x,y]$ is parallel to $N$ or the line passing through $x,y$ intersects $N$ but $[x,y] \cap N = \varnothing$.\\
 In this case, we can take $w:= x$.\\
 
 {\bf Case 2:} $[x,y] \cap N \not = \varnothing$.\\
  Since $\codim N \ge 2$, there exists a hyperplane $\tilde{N}$ containing $x,y, N$. Let $w_0= [x,y] \cap N$.  Let $w$ be a point in $\R^m \backslash \tilde{N}$ so that $|w-w_0|= \frac{R}{r}$ and $[w,w_0]$ is orthogonal to $\tilde{N}$. We have
$$|w-x|+ |w-y| \le |x- w_0|+ 2 |w-w_0|+ |y-w_0| \le |x-y|+ 2R \le 3 |x-y|,$$
where the last estimate holds due to (\ref{ine-sosanhxyR}). 

Moreover, if $\xi\in [x, w]\cup [y, w]$ then
$$\dist( \xi, [x, y])\leq \dist (w, [x, y])=|w-w_0|=\frac{R}{r}.$$
Let $x_0\in N$ such that $|x-x_0|=\dist (x, N)$. If $\xi\in [x, w]$ then $\xi=\alpha x+(1-\alpha)w$ for some $\alpha\in [0, 1]$. 
Since $x-x_0\perp N$ and $w-w_0\perp \tilde{N}$, we have
$$\dist (\xi, N)=|\xi-\alpha x_0-(1-\alpha)w_0|=\sqrt{\alpha^2|x-x_0|^2+(1-\alpha)^2|w-w_0|^2}\geq\frac{R}{2r}.$$
Similarly, if $\xi\in [y, w]$ then we also have $\dist (\xi, N)\geq\frac{R}{2r}.$
Then, $w$ satisfies the desired properties.

{\bf Case 3:}  $[x,y]$ is not parallel to $N$ and the line $d$ passing through $x,y$ does not intersect $N$.\\
 In this case, there exist $w_1 \in N$, $w_2 \in d$ such that $[w_1,w_2]$ is orthogonal to $N$ and $d$, and $|w_1-w_2|= \min\{|z-z'|: z \in d, z' \in N\}$.  Using the convexity of $d(\cdot, N)$ and the fact that there exists $\eta \in [x,y]$ with $d(\eta, N) < R/2$, we deduce that $d(\xi, N) > R/2$ for every $\xi \in d \backslash [x,y]$. Consequently $w_2 \in [x,y]$.

Let $w$ be the point in the line passing through $w_1,w_2$ such that $w_2$ lies between $w_1$ and $w$, and $|w-w_2|= R/r$. We check that $w$ satisfies the required properties. 
Let $\xi \in [w,x]$. Write $\xi= \alpha x+ (1- \alpha) w$ for some constant $\alpha \in [0,1]$. Let $\xi_0, x_0$ be points in $N$ such that $[x,x_0]$ and $[\xi,\xi_0]$ are orthogonal to $N$. We have 
$$\xi_0= \alpha x_0+ (1- \alpha) w_1.$$
Compute
\begin{align*}
|\xi- \xi_0|^2 &= |\alpha(x-x_0)+ (1- \alpha)(w-w_1)|^2 \\
&= \alpha^2 |x-x_0|^2+ (1- \alpha)^2 |w-w_1|^2+ 2 \alpha (1- \alpha) \langle x-x_0, w-w_1\rangle. 
\end{align*}
Recall that $w-w_1$ is both orthogonal to $N$ and $d$. It follows that 
$$\langle x-x_0, w-w_1 \rangle = \langle x- w_2, w-w_1 \rangle + \langle w_2- w_1, w- w_1 \rangle + \langle w_1- x_0, w-w_1 \rangle$$ 
which is equal to $\langle w_2- w_1, w- w_1 \rangle \ge 0$. Hence we obtain 
\begin{align*}
|\xi- \xi_0|^2 & \ge   \alpha^2 |x-x_0|^2+ (1- \alpha)^2 |w-w_1|^2 \\
& \ge  \alpha^2 R^2+ (1- \alpha)^2 R^2/ r^2 \ge \frac{R^2}{2r^2} \cdot
\end{align*}
Since $\dist(\xi,N)= |\xi-\xi_0|$, we infer 
$$2\dist(\xi, N) \ge R/r.$$
On the other hand, we have
$$\dist(\xi, [x,y]) \le \dist(w,[x,y]) \le |w-w_2|= R/r.$$
We obtain a similar inequalities if $\xi \in [w,y]$. Finally, observe 
$$|w-x|+ |w-y| \le  |w- w_2|+ |x-w_2|+ |w-w_2|+ |w_2-y|= 2R+ |x-y| \le 3R \le 3|x-y|,$$
because $w_2 \in [x,y]$ and we used here (\ref{ine-sosanhxyR}).   Thus $w$ satisfies the desired properties.

  This finishes the proof. 
\endproof

\begin{proof}[Proof of Lemma \ref{le-chonduongcong}]
We will use induction in $k$.
The case $k=1$ is an immediate corollary of Lemma \ref{le-chonduongcongA}. Assume that Lemma \ref{le-chonduongcong}
is true for $k=k_0$. We will show that it is also true for $k=k_0+1$.

Denote $N'=N_1\cup N_2\cup ...\cup N_{k_0}$ and $N=N_1\cup N_2\cup...\cup N_{k_0+1}$. 
Let $x, y\in\R^m\setminus N$, $x\neq y$. By the induction assumption, there exists  a polygonal chain
$l_0=a_0a_1...a_{4^{k_0}}$ with $a_0=x$ and $a_{4^{k_0}}=y$ such that
\begin{equation}\label{eq1le-chonduongcong}
	L(l_0)\leq C_0|x-y|,
\end{equation}
and
\begin{equation}\label{eq2le-chonduongcong}
	C_0 \dist (\xi, N')\geq \min\{\dist(x, N'), \dist(y, N')\},
\end{equation}
for every $\xi\in \cup_{s=1}^{4^{k_0}}[a_{s-1}, a_s]$, where $C_0\geq 1$
 is a constant depending only on $k_0$ and $m$. 

We will construct a polygonal chain $l=b_0b_1...b_{4^{k_0+1}}$ satisfying the conditions as in Lemma \ref{le-chonduongcong}.
Denote 
$$A:= \min\{\dist(x, N), \dist(y, N)\}.$$
 If $2C_0 \dist (\xi, N)\geq A$
for every $\xi\in \cup_{s=1}^{4^{k_0}}[a_{s-1}, a_s]$ then we can choose $l=l_0$ and $C=2C_0$.
 It remains to consider the case where $2C_0 \dist (\xi_0, N)< A$ for some 
$\xi_0\in \cup_{s=1}^{4^{k_0}}[a_{s-1}, a_s]$. In this case, we have
\begin{equation}\label{eq3le-chonduongcong}
L(l_0)\geq |x-\xi_0|+|y-\xi_0|\geq (\dist (x, N)-\dist (\xi_0, N))+(\dist (y, N)-\dist (\xi_0, N))\geq A.
\end{equation}
For every $s=0,..., 4^{k_0}$, we define $b_{4s}$ as follows
\begin{itemize}
	\item If $2C_0 \dist (a_s, N)\geq A$ then we put $b_{4s}=a_s$;
	\item If  $2C_0 \dist (a_s, N)< A$ then we choose $b_{4s}\in\R^m$ such that the vector $b_{4s}-a_s$ is perpendicular to $N_{k_0+1}$ and 
	\begin{equation}\label{eq4le-chonduongcong}
	\dist (b_{4s}, N_{k_0+1})=|a_s-b_{4s}|+\dist (a_s, N_{k_0+1})=\frac{A}{2C_0}.
	\end{equation}
\end{itemize}
Thus we have
\begin{equation}\label{eq5le-chonduongcong}
|b_{4s}-b_{4s+4}|\leq |a_s-a_{s+1}|+|a_s-b_{4s}|+|a_{s+1}-b_{4s+4}|\leq |a_s-a_{s+1}|+\frac{A}{C_0},
\end{equation}
and
\begin{equation}\label{eq6le-chonduongcong}
\dist (\xi, [a_s, a_{s+1}])\leq\max\{|b_{4s}-a_s|, |b_{4s+4}-a_{s+1}| \}\leq\frac{A}{2C_0},
\end{equation}
for each $\xi\in [b_{4s}, b_{4s+4}]$ and for every $s=0, 1, ..., 4^{k_0}-1$. 

Combining \eqref{eq2le-chonduongcong} and \eqref{eq6le-chonduongcong}, we get
\begin{equation}\label{eq7le-chonduongcong}
\dist (\xi, N')\geq \inf_{\eta\in [a_s, a_{s+1}]}\dist(\eta, N')-\dist (\xi, [a_s, a_{s+1}])\geq \frac{A}{2C_0},
\end{equation}
for each $\xi\in [b_{4s}, b_{4s+4}]$ and for every $s=0, 1, ..., 4^{k_0}-1$. 

We will find $b_{4s+1}, b_{4s+2}$ and $b_{4s+3}$ such that
\begin{itemize}
	\item [(i)] $\sum_{j=4s}^{4s+3}|b_j-b_{j+1}|\leq 3|a_s-a_{s+1}|+\frac{2A}{C_0}$;
	\item [(ii)] $\dist (\xi, N)\geq\frac{A}{8C_0}$ for every $\xi\in \cup _{j=4s}^{4s+3}[b_j, b_{j+1}]$.
\end{itemize}
We distinguish into three cases.\\

\noindent
{\bf Case 1:}  $\dist (\xi, N_{k_0+1})\geq\frac{A}{4C_0}$ for all $\xi\in [b_{4s}, b_{4s+4}]$.\\
 In this case, we put  $b_{4s+1}= b_{4s+2}=b_{4s+3}=b_{4s+4}$. It follows from
  \eqref{eq5le-chonduongcong} and \eqref{eq7le-chonduongcong} that the conditions $(i)$ and $(ii)$ are satisfied.\\

 \noindent
 {\bf Case 2:}  $\dist (\xi_0, N_{k_0+1})<\frac{A}{4C_0}$ for some $\xi_0\in [b_{4s}, b_{4s+4}]$ and either $a_s\neq b_{4s}$
 or $a_{s+1}\neq b_{4s+4}$.\\
 In this case, we have
 $$\min\{\dist (b_{4s}, N_{k_0+1}), \dist (b_{4s+4}, N_{k_0+1})\}=\frac{A}{2C_0}.$$
 By Lemma \ref{le-chonduongcongA}, we can choose $q\in \R^n$ such that
 \begin{equation}\label{eq7.1le-chonduongcong}
 |b_{4s}-q|+|b_{4s+4}-q|\leq 3|b_{4s}-b_{4s+4}|
 \end{equation}
 and 
 \begin{equation}\label{eq7.2le-chonduongcong}
 4\dist (\xi, N_{k_0+1})\geq\frac{A}{2C_0}\geq 2\dist (\xi, [b_{4s}, b_{4s+4}]),
 \end{equation}
 for every $\xi\in [b_{4s}, q]\cup [q, b_{4s+4}]$.
 
 By \eqref{eq7le-chonduongcong} and \eqref{eq7.2le-chonduongcong}, we have
 \begin{equation}\label{eq7.3le-chonduongcong}
 \dist (\xi, N')\geq \inf_{\eta\in [b_{4s}, b_{4s+4}]}\dist (\eta, N')-\dist (\xi, [b_{4s}, b_{4s+4}])\geq\frac{A}{4C_0}.
 \end{equation}
 for every $\xi\in [b_{4s}, q]\cup [q, b_{4s+4}]$.
 
 Put $b_{4s+1}=b_{4s+2}=b_{4s+3}=q$.
It follows from \eqref{eq5le-chonduongcong} and \eqref{eq7.1le-chonduongcong} that $(i)$ is satisfied.
 By \eqref{eq7.2le-chonduongcong} and \eqref{eq7.3le-chonduongcong}, we also get $(ii)$.\\
 
 \noindent
 {\bf Case 3:}  $\dist (\xi_0, N_{k_0+1})<\frac{A}{4C_0}$ for some $\xi_0\in [b_{4s}, b_{4s+4}]$ and  $a_j=b_{4j}$
for $j=s, s+1$.\\
	 In this case,
  we choose $b_{4s+2}\in\R^m$ such that the vector $b_{4s}-a_s$ is perpendicular to $N_{k_0+1}$ and 
	\begin{equation}\label{eq8le-chonduongcong}
	\dist (b_{4s+2}, N_{k_0+1})=|b_{4s+2}-b_{4s}|+\dist (\xi_0, N_{k_0+1})=\frac{A}{4C_0}.
	\end{equation}
	Similar to \eqref{eq5le-chonduongcong} and \eqref{eq6le-chonduongcong} 
	(and note that $a_j=b_{4j}$ for $j=s, s+1$), we have
	\begin{equation}\label{eq9le-chonduongcong}
	|b_{4s}-b_{4s+2}|+|b_{4s+2}-b_{4s+4}|\leq |b_{4s}-b_{4s+4}|+\frac{A}{2C_0}=|a_s-a_{s+1}|+\frac{A}{2C_0},
	\end{equation}
	and 
	\begin{equation}\label{eq10le-chonduongcong}
\dist (\xi, [a_s, a_{s+1}])=	\dist (\xi, [b_{4s}, b_{4s+4}])\leq \frac{A}{4C_0},
	\end{equation}
	for every $\xi\in [b_{4s}, b_{4s+2}]\cup [b_{4s+2}, b_{4s+4}]$.
	
	Combining \eqref{eq2le-chonduongcong} and \eqref{eq10le-chonduongcong}, we get
	\begin{equation}\label{eq10.1le-chonduongcong}
	\dist (\xi, N')	\geq \frac{3A}{4C_0},
	\end{equation}
	for every $\xi\in [b_{4s}, b_{4s+2}]\cup [b_{4s+2}, b_{4s+4}]$.

By using Lemma \ref{le-chonduongcongA} for $[b_{4s}, b_{4s+2}]$ and $[b_{4s+2}, b_{4s+4}]$,
 we can choose $b_{4s+1}$ and $b_{4s+3}$ such that
 \begin{equation}\label{eq11le-chonduongcong}
 \sum_{j=4s}^{4s+3}|b_j-b_{j+1}|\leq 3(|b_{4s}-b_{4s+2}|+|b_{4s+2}-b_{4s+4}|),
 \end{equation}
	and
\begin{equation}\label{eq12le-chonduongcong}
2\dist (\xi, N_{k_0+1})\geq\frac{A}{4C_0}\geq \dist (\xi, [b_{4s}, b_{4s+2}]\cup [b_{4s+2}, b_{4s+4}]),
\end{equation}
for every $\xi\in \cup_{j=4s}^{4s+3}[b_j, b_{j+1}]$.

By \eqref{eq10.1le-chonduongcong} and \eqref{eq12le-chonduongcong}, we have
\begin{equation}\label{eq13le-chonduongcong}
\dist (\xi, N')\geq\inf_{\eta\in  I}\dist (\eta, N')
-\dist (\xi,  I)
\geq \frac{A_0}{2C_0},
\end{equation}
for every $\xi\in \cup_{j=4s}^{4s+3}[b_j, b_{j+1}]$, where $I=[b_{4s}, b_{4s+2}]\cup [b_{4s+2}, b_{4s+4}]$.
It follows from \eqref{eq12le-chonduongcong} and \eqref{eq13le-chonduongcong} that $(ii)$ is satisfied. By \eqref{eq9le-chonduongcong}
and \eqref{eq11le-chonduongcong}, we also obtain $(i)$.

Now, let $l_0=b_0b_1...b_{4^{k_0+1}}$. By $(ii)$, we have
$$\dist (\xi, N)\geq\frac{1}{8C_0}\min\{\dist(x, N), \dist (y, N)\}.$$
By \eqref{eq1le-chonduongcong}, \eqref{eq3le-chonduongcong} and $(i)$, we have
\begin{align*}
L(l)=\sum_{j=0}^{4^{k_0+1}-1}|b_j-b_{j+1}|&\leq\sum_{s=0}^{4^{k_0}-1}\left(3|a_s-a_{s+1}|+\frac{2A}{C_0}\right)\\
&=3 L(l_0)+4^{k_0}\frac{2A}{C_0}\\
&\leq 3\left(1+\frac{4^{k_0}}{C_0}\right)L(l_0)\\
&\leq 3\left(C_0+4^{k_0}\right)|x-y|.
\end{align*}
Choosing $C=8\left(C_0+4^{k_0}\right)$, we see that $l$ and $C$ satisfy the desired conditions.
Thus, Lemma \ref{le-chonduongcong} is true in the case $k=k_0+1$.
This completes the proof.
\end{proof}

\section{Log continuity preserved under blowups} \label{sec-log2}

%We begin with several observations about Riemannian metrics. 
Let $f: X \to Y$ be a smooth surjective map between compact differentiable manifolds. Let $g_X, g_Y$ be Riemannian metrics on $X,Y$ respectively. Let $d_{g_X},d_{g_Y}$ denote the distances induced by $g_X, g_Y$ on $X,Y$ respectively. For $E \subset X$, let $d_X(a, E):= \inf_{b \in E} d_X(a,b)$.  For every $a,b \in Y$, we define 
$$d_{g_X, f}(a,b):= \inf_{a' \in f^{-1}(a), b' \in f^{-1}(b)}d_{g_X}(a',b').$$
We note the last function is in general not a metric on $Y$.
%$$d_{g_X, f}(a,b):= \max\{ \sup_{a' \in f^{-1}(a)}d_X(a, f^{-1}(b)), \sup_{b' \in f^{-1}(b)}d_X(b', f^{-1}(a))\}$$
Observe 
\begin{align} \label{ine-sosanhdfiber}
d_{g_Y} \le C d_{g_X,f}
\end{align}
for some constant $C>0$ because the differential $Df$ is bounded uniformly on $X$. % One can see also that for sequences $(a_j)_j, (b_j)_j \subset Y$, if $d_{g_Y}(a_j,b_j) \to 0$, then $d_{g_X, f}(a_j, b_j) \to 0$ as $j \to \infty$. 

\begin{lemma} \label{le-successiveblowup} Let $X_0, \ldots, X_m$ be compact complex manifolds and assume that $f:X_j \to X_{j-1}$ is the blow up along a smooth submanifold $V_{j-1} \subset X_{j-1}$ in $X_{j-1}$ for $1 \le j \le m$. Denote $f:= f_m \circ \cdots \circ f_0: X_m \to X_0$. Let $g_j$ be a Riemannian metric on $X_j$ for $1 \le j \le m$.    Then,
 for every  function  $u: X_0\rightarrow\R$, if $u \circ f_m$ is  a  $\log^M$-continuous function for some $M>0$, then so is $u$.
\end{lemma}

We note that a similar property for H\"older continuity was proved in \cite{BinGuo}. The following proof is more or less similar. 

\proof This is indeed implicitly in the proof of Lemma 3.4 in \cite{BinGuo}  if $m=1$. The general case follows from an immediate induction on $m$. For readers' convenience, we reprove below the case where $m=1$. 

Let $f_1: X_1 \to X_0$ be the blow up along a smooth submanifold $V$ in $X_0$. Denote $n:= \dim X_0$ and $l:= \dim V$. Asume $\xi\in V$ and let $\big(U, x=(x_1, \ldots, x_n)\big)$ be a local chart around $\xi$ such that $V$ is given by $\{x_j=0, 1\le j \le n-l\}$. Thus $f_1^{-1}(U)$ is biholomorphic to the submanifold of $U \times \C\P^{n-l-1}$ defined by the equations  $x_j v_s= v_j x_s$ for $1\le j, s \le n-l$, where $v:=[v_1: \cdots:v_{n-l}]$ are the homogeneous coordinates on $\C\P^{n-l-1}$. One can cover $f^{-1}_1(U)$ by $(n-l)$ open subsets 
$$U_j:= \big \{(x, v) \in f_1^{-1}(U): v_j \not =0\big\}.$$ 
  In $U_j$, we have
$$f_1(x, v)= (v_1 x_j/v_j, v_2x_j/v_j, \ldots, v_{n-l}x_j/v_j,  x_{n-l+1}, \ldots, x_n).$$

Now let $a,b \in X_0$. It suffices to consider $a,b$ close to each other and both close to $V$ (because $f_1$ is biholomorphic outside $V$).  We split the proof into several cases. Firstly observe that if $a,b \in V$, then since $f_1:f_1^{-1}(V) \to V$ is a submersion, one gets 
$$C d_{g_0}(a,b) \ge d_{g_1,f_1}(a,b),$$
for some constant $C>0$ independent of $a$ and $b$. Hence 
\begin{align}\label{ine-uholder0chuyenf}
|u(a)- u(b)|&= \inf_{a' \in f_1^{-1}(a), b' \in f_1^{-1}(b)} |u \circ f_1(a')- u \circ f_1(b')| \\
\nonumber
&\lesssim |\log d_{g_1,f_1}(a,b)|^{-M} \lesssim |\log d_{g_0}(a,b)|^{-M}.
\end{align}
Note that in the last inequality, we only consider $a$ and $b$ close to each other, hence $\log d_{g_0}(a,b)<0$.
\\

\noindent
\textbf{Case 1.} Consider now $b \in V$ and $a \not \in V$ but close to $b$. Then there is a local chart $(U, x)$ on $X_0$ containing $b,a$ such that $V$ is given by $\{x_j=0, 1\le j \le n-l\}$. We use now the Euclidean distance on that local chart.  

Without loss of generality, we can assume that $b=0$, $a=(x'_{1}, \ldots, x'_{n})$ with $x'_1\neq 0$ and $h(t):=(t x'_{1}, \ldots, t x'_{n})\in U$ for every $t \in [0,1]$.    We see that 
$$\hat{h}(t):=f_1^{-1} \circ h(t)= (t x'_1,\ldots, tx'_n, [x'_1:\ldots: x'_n]),$$
for $t>0$. Letting $t \to 0$ gives 
$$\hat{h}(0):=\lim_{t \to 0} f_1^{-1} \circ h(t)= (0,\ldots, 0, [x'_1:\ldots: x'_n]) \in f_1^{-1}(b).$$
We infer that
$$d_{g_1}\big(\hat{h}(1),\hat{h}(0)\big) \lesssim  |x'_{1}|+ \cdots + |x'_{n}| \lesssim |a-b|.$$
It follows that 
$$d_{g_1}\big(\hat{h}(1), f^{-1}(b)\big) \lesssim |a-b| \lesssim d_{g_0}(a,b).$$ 
Hence $d_{g_1,f_1}(a,b) \lesssim d_{g_0}(a,b)$. Thus we get an estimate similar to (\ref{ine-uholder0chuyenf}).
\\

\noindent
\textbf{Case 2.} Consider now $a,b \not \in V$ but close to $V$. Direct computations show  that $|D f^{-1}_1(a)| \lesssim |d_{g_0}(a, V)|^{-2}$. Thus we get 
$$d_{g_1,f_1}(a,b) \lesssim \max \{d_{g_0}(a,V)^{-2}, d_{g_0}(b,V)^{-2}\} d_{g_0}(a,b).$$  
Hence if $\min \{d_{g_0}(a,V)^{2}, d_{g_0}(b,V)^{2}\} \ge d_{g_0}(a,b)^{1/2}$, then 
$$d_{g_1,f_1}(a,b) \lesssim d_{g_0}(a,b)^{1/2}.$$
We treat now the case where $\min \{d_{g_0}(a,V)^{2}, d_{g_0}(b,V)^{2}\} \le d_{g_0}(a,b)^{1/2}$. Without loss of generality, we can assume that $d_{g_0}(b, V) \le d_{g_0}(a,b)^{1/4}$. Then
$$d_{g_0}(a,V) \le d_{g_0}(a,b)+ d_{g_0}(b,V) \lesssim d_{g_0}(a,b)^{1/4}.$$ 
Now we consider a local chart $(U,x)$ containing $a,b$. We use now the Euclidean metric. Let $a_V, b_V$ be the projection of $a,b$ to $V$ respectively. % Let $c_V$ be the mid-point of the interval $[a_V, b_V]$.
  Observe that 
$$|a_V- b_V| \lesssim |a-b|$$
and 
$$|a-a_V| \lesssim |a-b|^{1/4}, |b-b_V| \lesssim |a-b|^{1/4}.$$
%Hence 
%$$|a-c_V| \lesssim |a-b|^{1/4}, |b- c_V| \lesssim |a-b|^{1/4}$$
Now applying Case 1 to $(a,a_V), (b,b_V)$ and $(a_V, b_V)$, one obtains 
\begin{align}\label{ine-abaVbV}
d_{g_1,f_1}(a,a_V)+d_{g_1,f_1}(a_V,b_V)+d_{g_1,f_1}(b_V,b) \lesssim |a-a_V|+ |a_V- b_V|+ |b-b_V| \lesssim |a-b|^{1/4}.
\end{align}
Put $x_1:= |u(a)- u(a_V)|$, $x_2:= |u(a_V)- u(b_V)|$ and $x_3:= |u(b_V)- u(b)|$.
By previous parts of the proof, we see that 
$$d_{g_0}(a, a_V) \gtrsim e^{-x_1^{-\frac{1}{M}}}, \quad d_{g_0}(a_V,b_V) \gtrsim e^{-x_2^{-\frac{1}{M}}},$$
and
$$ \quad d_{g_0}(b_V,b) \gtrsim e^{-x_3^{-\frac{1}{M}}}.$$
 This combined with (\ref{ine-sosanhdfiber}) and (\ref{ine-abaVbV}) gives
$$|a-b|^{1/4} \gtrsim \sum_{j=1}^3  e^{-x_j^{-\frac{1}{M}}}
\gtrsim \exp\left\{-\left(\dfrac{x_1+x_2+x_3}{3}\right)^{-1/M}\right\}.$$
%Arguing as in (\ref{ine-uholder0chuyenf}), one sees that the left-hand side of the last inequality is
%$$ \gtrsim |u(a)- u(a_V)|^{1/\beta}+ |u(a_V)- u(b_V)|^{1/\beta}+ |u(b_V)- u(b)|^{1/\beta} \gtrsim |u(a)- u(b)|^{1/\beta}$$ by the triangle inequality. 
%We thus infer that $$|u(a)- u(b)| \lesssim d_{g_0}(a,b)^{\beta/4}.$$
%This finishes the proof.

It follows that 
$$x_1+x_2 +x_3 \lesssim \left|\log \dfrac{|a-b|}{C}\right|^{-M}\lesssim |\log |a-b||^{-M},$$
for some constant $C>0$ independent of $a$ and $b$.
The left-hand side of the last inequality is $\ge |u(a)- u(b)|$. Hence $|u(a)- u(b)| \lesssim |\log |a-b||^{-M}$. This finishes the proof.
%In a local chart $(U,x)$, the right-hand side of the last inequality is equivalent to $|a-b|/(|x|)$
\endproof

\section{ H\"older continuous measures}  \label{sec-capa}

Let $\eta$ be a closed smooth semi-positive $(1,1)$-form in a big (semi-positive) cohomology class. Let $K$ be a Borel subset of $X.$ The \emph{capacity} of $K$ is given by 
$$\capK_\eta(K):= \sup \big\{ \int_K \eta_\varphi^n: 0 \le \varphi \le 1, \varphi \, \text{ $\eta$-psh} \big \}.$$
The above notion was introduced in \cite{EGZ} generalizing those in \cite{Bedford_Taylor_82,Kolodziej_Acta}; see \cite{Lu-Darvas-DiNezza-logconcave,Lu-comparison-capacity} and references therein for various generalizations of capacity. 

Let $A, \beta>0$. We say that a Borel measure $\mu$ on $X$ satisfies  the condition $\mathcal{H}(\beta, A, \eta)$ if 
$$\mu(K)\leq A \left(\capK_\eta(K)\right)^{1+\beta},$$
for every Borel set $K\subset X$. 

%We conclude this section with some facts about H\"older continuous measures.  
Fix a K\"ahler form $\omega$ on $X$. Let $\mu$ be a measure on $X$. Recall that $\mu$ is said to be a H\"older continuous measure with the H\"older constant $A$ and the H\"older exponent $\gamma \in (0,1]$ if for every $\omega$-psh functions $\varphi_1,\varphi_2$ with $\int_X \varphi_j \omega^n=0$ for $j=1,2$ there holds 
$$\int_X (\varphi_1- \varphi_2) d \mu \le A \|\varphi_1- \varphi_2\|_{L^1(\omega^n)}.$$

Let $\mathcal{M}(A,\gamma)$ be the set of H\"older continuous measures with the H\"older constant $A$ and the H\"older exponent $\gamma \in (0,1]$. By \cite[Lemma 3.3]{DinhVietanhMongeampere},  a measure $\mu \in \mathcal{M}(A, \gamma)$  if there is a constant $C>0$ depending only on $A$ such that for every $\omega$-psh functions $\varphi_1,\varphi_2$, we have
\begin{align} \label{le_ine_trunhau}
	\int_X|\varphi_1 - \varphi_2| d\mu \le C \max \big \{ \|\varphi_1 -\varphi_2\|_{L^1(X)}^{\gamma},\|\varphi_1 -\varphi_2\|_{L^1(X)} \big \}.
\end{align} 
Note that if $\mu$ is a H\"older continuous measure then it follows from 
\cite[Proposition 2.4 and Proposition 4.4]{DinhVietanhMongeampere} that for every constant $\beta>0$, there exists  a constant $A_{\beta}>0$ such that
$\mu$ satisfies  the condition $\mathcal{H}(\beta, A_{\beta}, \omega)$. Therefore, by the comparison of capacities (see \cite[Theorem 3.17]{Vu_DoHS-quantitative}),
for every $\beta>0$, there exists $A_{\beta}>0$ such that
$\mu$ satisfies  the condition $\mathcal{H}(\beta, A_{\beta}, \eta)$. Alternatively, one can prove the last property by using results in \cite{Lu-Darvas-DiNezza-logconcave}.

The following lemma is a special case of \cite[Proposition 5.3]{Guedj-Zeriahi-big-stability} (replace $\varphi$ and $\psi$ by
$w_1$ and $w_2$, respectively):

\begin{lemma}\label{lem-chanduoidayulambda} Let $\eta$ be a big semi-positive closed smooth $(1,1)$-form and let
	$w_1,w_2$ be negative $\eta$-psh functions such that $w_1$ is  of full Monge-Amp\`ere mass (i.e, $\int_X \eta_{w_1}^n= \int_X \eta^n$).
	Denote  $\mu_1=(\eta+dd^cw_1)^n$. Assume that the following conditions hold
	\begin{itemize}
		\item [(i)] there exists a constant $M>0$ such that
		$$-M\leq \max\{w_1, w_2\}\leq 0;$$
		\item [(ii)] there exist constants $A, \beta>0$ such that $\mu_1$ satisfies the condition $\mathcal{H}(\beta, A, \eta)$.
	\end{itemize}
	Then, for every constant $r>0$, there exists a constant $C>0$ depending on $\omega, \eta, M, A, \beta$ and $r$ such that
	$$w_1- w_2  \ge  -C  \|w_1- w_2\|_{L^r(\mu_1)}^{\frac{\beta r}{n+\beta(n+r)}}.$$
 In particular, if $\mu_1\in\mathcal{M}(B,\alpha)$ for some $B>0$ and $0<\alpha\leq 1$ then
 for every $r, \gamma>0$, there exists a constant $C'>0$
 depending on $\omega, \eta, B, \alpha, \gamma$ and $r$ such that
 	$$w_1- w_2  \ge  -C'  \|w_1- w_2\|_{L^r(\mu_1)}^{\frac{\gamma r}{n+\gamma(n+r)}}.$$
\end{lemma}

In the next sections, we will apply Lemma \ref{lem-chanduoidayulambda} to the case where $r$ is large enough, this means the exponent $\frac{\beta r}{n+\beta(n+r)}$ is close to be 1. 

By using \eqref{le_ine_trunhau} and applying Lemma \ref{lem-chanduoidayulambda} to the case where $w_1=w$ and $w_2=0$, we obtain the following corollary:

\begin{corollary}\label{cor-boundosc}
	Let $\eta$ be a big semi-positive closed smooth $(1,1)$-form and let
	$w$ be a negative $\eta$-psh function of full Monge-Amp\`ere mass with $\sup_X w =0$. Assume that
	$(\eta+dd^cw)^n\in\mathcal{M}(B,\alpha)$ for some $B>0$ and $0<\alpha\leq 1$. Then $\|w\|_{L^\infty}\leq C$,
	where $C>0$ is a constant depending on $ \omega, \eta, B$ and $\alpha$.
\end{corollary}

By \cite{DinhVietanhMongeampere}, a measure $\mu$ of mass $\int_X \omega^n$ is H\"older continuous if and only if $\mu= (\ddc u+ \omega)^n$ for some H\"older continuous $\omega$-psh function $u$ on $X$. 
The following will be important for us. 

\begin{corollary}\label{cor-pushdownMAolder} Let $X_0, \ldots, X_m$ be compact complex manifolds and assume that $f_j:X_j \to X_{j-1}$ is the blow up along a smooth submanifold $V_{j-1} \subset X_{j-1}$ in $X_{j-1}$ for $1 \le j \le m$. Denote $f:= f_m \circ \cdots \circ f_0: X_m \to X_0$.  Suppose that $\mu$ is a H\"older continuous measures on $X_m$. Then $f_* \mu$ is also H\"older continuous. 
\end{corollary}

\proof By induction, it suffices to prove the desired assertion for $m=1$. Let $u_1,u_2$ be $\omega_0$-psh functions on $X_0$ for $j=1,2$, where $\omega_0$ is a K\"ahler form on $X_0$. Put $u'_j:= f_1^* u_j$. Let $\omega_1$ be  a K\"ahler form on $X_1$. Using H\"older continuity of $\mu$, we obtain
$$\|u_1- u_2\|_{L^1((f_1)_* \mu)}= \|u'_1- u'_2\|_{L^1(\mu)} \lesssim \|u'_1- u'_2\|^\gamma_{L^1(\omega_1^n)}+ \|u'_1- u'_2\|_{L^1(\omega_1^n)}.$$  
Standard computations using local coordinates for blowups show that there exists a function $g \in L^p(\omega_0^n)$ for some constant $p>1$ satisfying $(f_1)_* \omega_1^n = g  \omega_0^n$. Hence 
$$\|u'_1- u'_2\|_{L^1(\omega_1^n)}= \int_{X_0} |u_1- u_2| g \omega_0^n \lesssim \|u_1- u_2\|_{L^q(\omega_0^n)}$$
where $1/q+ 1/p =1$. By \cite[Lemma 2.2]{DinhVietanhMongeampere}, one has 
$$\|u_1- u_2\|_{L^q(\omega_0^n)} \lesssim \|u_1- u_2\|_{L^1(\omega_0^n)}^{1/(2q)}.$$
Hence $(f_1)_*\mu$ is H\"older continuous. 
\endproof

\section{Regularization of psh functions}\label{sec.appr}

\subsection{$L^2$-estimates}

We recall first the $L^2$-estimates for $\bar \partial$ and discuss some of its variants. 

\begin{theorem}\label{the-dbar0} (see \cite[Corollary 5.3]{Demailly_analyticmethod}) Let $(X,\omega)$ be a compact K\"ahler manifold. Let $\epsilon>0$ be a constant. Let $L$ be a holomorphic line bundle on $X$ together with a singular Hermitian metric $h$ satisfying 
\begin{align*}
 c_1(L,h) \ge \epsilon \omega.
 \end{align*}
Then for every $g \in L^2_{n,1}(X, L)$ with $\bar \partial g=0$, there exists $u \in L^2_{n,0}(X,L)$ such that $\bar \partial u =g$, and 
$$\int_X |u|^2_{h,\omega} \omega^n \le \epsilon^{-1} \int_X |g|^2_{h,\omega} \omega^n,$$
where $|g(x)|_{h,\omega}$ denotes the norm of $g$ with respect to the norm induced by the Hermitian metric $h$ on $L$ and the Riemannian metric on $X$  associated to $\omega$. 
%Moreover if $g$ is smooth, then we can choose $u$ smooth. 
\end{theorem}

We deduce from the above result the following more or less standard consequence. %; for example, see \cite[Corollary 6.7]{Boucksom_l2}. 

\begin{theorem}\label{the-dbar} Let $(X,\omega)$ be a compact K\"ahler manifold. Let $\epsilon>0$ be a constant. Let $K^*_X$ be the dual of the canonical line bundle, and  let $h_{K^*_X}$ denote the metric induced by $\omega$ on $K^*_X$. Let $L$ be a holomorphic line bundle on $X$ together with a singular Hermitian metric $h$. Assume that there exists a singular metric $\tilde{h}_{K^*_X}$ on $K^*_X$ so that  
\begin{align*}
c_1(L,h)+ c_1(K^*_X, \tilde{h}_{K^*_X}) \ge \epsilon \omega.
 \end{align*}
Then for every $g \in L^2_{0,1}(X, L)$ with $\bar \partial g=0$, there exists $u \in L^2_{0,0}(X,L)$ such that $\bar \partial u =g$, and 
$$\int_X |u|^2_h e^{-2 \vartheta} \omega^n \le \epsilon^{-1} \int_X |g|^2_{h,\omega}  e^{-2 \vartheta}\omega^n,$$
where $\vartheta$ is a quasi-psh function defined by $\tilde{h}_{K^*_X}= e^{-2 \vartheta} h_{K^*_X}$.
%Moreover if $g$ is smooth, then we can choose $u$ smooth. 
\end{theorem} 

\proof Set $L':= L \otimes K^*_X$. Thus $L= L' \otimes K_X$.  Let $h'$ be the singular Hermitian metric on $L'$ given by $h'= h \otimes \tilde{h}_{K^*_X}$. For every $0\le q \le n,$ we have a natural isometry 
$$\Psi_q: \Lambda^{0,q}(T^*X) \otimes L \to  \Lambda^{n,q}(T^*X) \otimes L',$$
e.g., see the proof of \cite[Corollary 4.3]{Coman-Marinescu-L2-orbifold}), where we use the metric $h$ on $L$, and $h \otimes h_{K_X^*}$ on $L'$. The map $\Psi$ commutes with $\partial, \bar \partial$ operators.  Thus $\Psi_1(g) \in L^2_{n,1}(X, L')$ with $\bar \partial \Psi_1(g)=0$. Since $\Psi_1$ is an isometry, one gets 
$$|\Psi_1(g)|^2_{h'}=|\Psi_1(g)|^2_{h\otimes h_{K_X^*}} e^{-2 \vartheta}=   |g|^2_h e^{-2\vartheta}.$$
 The desired assertion now follows from Theorem \ref{the-dbar0}. 
\endproof

In particular we obtain the following.

%\begin{corollary}\label{cor-dbar} Let $(X, \omega)$ be a compact K\"ahler manifold so that the Chern class of $K^*_X$ is semi-positive. Let $L$ be a holomorphic line bundle on $X$ together with a singular Hermitian metric $h$ such that 
%\begin{align*}
% c_1(L,h) \ge \epsilon \omega.
% \end{align*}
%Then there is a positive constant $C>0$ such that for every $g \in L^2_{0,1}(X, L)$ with $\bar \partial g=0$, there exists $u \in L^2_{0,0}(X,L)$ such that $\bar \partial u =g$, and 
%$$\int_X |u|^2_h \omega^n \le  \frac{C}{\epsilon} \int_X |g|^2_{h,\omega} \omega^n.$$
%\end{corollary} 

\begin{corollary}\label{cor-dbar} Let $(X, \omega)$ be a compact K\"ahler manifold so that the Chern class of $K^*_X$ contains a closed positive $(1,1)$-current of bounded potentials, \emph{i.e,} there exists a bounded $\eta_\omega$-psh function $\vartheta$ on $X$, where $\eta_\omega$ is the Chern form of the metric on $K^*_X$ induced  by $\omega$.  Let $L$ be a holomorphic line bundle on $X$ together with a singular Hermitian metric $h$ such that 
\begin{align*}
 c_1(L,h) \ge \epsilon \omega.
 \end{align*}
Then for every $g \in L^2_{0,1}(X, L)$ with $\bar \partial g=0$, there exists $u \in L^2_{0,0}(X,L)$ such that $\bar \partial u =g$, and 
$$\int_X |u|^2_h \omega^n \le  \frac{e^{4 \|\vartheta\|_{L^\infty}}}{\epsilon} \int_X |g|^2_{h,\omega} \omega^n.$$
\end{corollary}

We recall now a special case of the Ohsawa-Takegoshi extension theorem (see \cite[Theroem 13.6]{Demailly_analyticmethod}). 

\begin{theorem} \label{th-OKYbang0} Let $X$ be a weakly pseudoconvex $n$-dimensional manifold with a K\"ahler metric $\omega$. Let $y$ be a point in $X$. Let $(L,h)$ be a line bundle with a smooth positively curved Hermitian metric $h$ on $X$ and let $E$ be the trivial holomorphic vector bundle of rank $n$ equipped with the trivial Hermitian metric such that there exists  a global section $s$ of $E$ with $y=\{s=0\}$ so that $\Lambda^n d s(y) \not =0$ and $|s| \le e^{-1}$. Then for every $(n,0)$-form $f$ with values in $L$ at $y$, there exist a $\bar \partial$-closed  $(n,0)$-form 
$F$ with values in $L$ on $X$ such that $F(y)=f(y)$ and
$$\int_X \frac{|F|^2_{h,\omega}}{|s|^{2n} (-\log |s|)^2} \omega^n \le C_n \frac{|f|^2_{h,\omega}}{|\Lambda^n d s(y)|_{\omega}^2},$$
where $C_n$ is a numerical constant depending only on $n$.
\end{theorem}

We note that since $E$ and its Hermitian metric are trivial, the curvature of the metric of $E$ vanishes everywhere, and $s$ is nothing but a tuple of $n$ holomorphic functions on $X$.

\begin{corollary} \label{cor-OKYbang0truocda} Let $X, \omega,L,h,E,y,s$ be as in Theorem \ref{th-OKYbang0}. Assume furthermore that $\Lambda^n d s(y) \not =0$ for every $y\in X\backslash N$, where $N$ is an analytic set in $X$.  Then for every section $f$ of $L$ at $y$, there exists  
$F \in H^0(X\backslash N,L)$ such that $F(y)=f(y)$ and
$$\int_X |F|^2_{h}|\Lambda^n d s(y)|_{\omega}^2 \omega^n \le C_n  |f|^2_{h},$$
where $C_n$ is a numerical constant depending only on $n$.
\end{corollary}

\proof Applying Theorem \ref{th-OKYbang0} to $f_1:= f\Lambda^n d s$, we find a $\bar\partial$-closed $(n,0)$-form $F_1$ with values in $L$ such that 
$$\int_X \frac{|F|^2_{h,\omega}}{|s|^{2n} (-\log |s|)^2} \omega^n \le C_n |f|^2_{h}.$$
Put $F:= F_1/ \Lambda^n d s$. By the hypothesis on $s$ we see that $F$ is holomorphic on $X \backslash N$ and satisfies the desired inequality. 
\endproof

%It is a good moment to mention a result about the extension of holomorphic functions which is used later in the paper: every holomorphic function on the complement of an analytic subset of codimension at least 2 in a normal complex space is automatically extended to a global holomorphic function on that space (see \cite{Grauert-Remmert-coherentsheaves}).

%WRITE A VERSION WITH SEMI-POSITIVE FORM AS WHAT WE DID FOR L2 estimate. NOTICE THAT THE NORM OF $\Lambda^n d s_E(y)$ with respect to $\theta_\epsilon$ converges to that of $\Lambda^n d s'_E(y)$ with respect to $\theta$ in $U'$, where $s_E=s'_E \circ \Phi$. Hence this norm is bounded away from  0 uniformly regardless how close $y$ to $N$ is.

%IF THIS COROLLARY WORKS; THEN WE CAN JUST APPLY IT TO A LOCAL CHART CONTAINING $y$ (of radius 1, hence this local chart could intersect $N$, but this is no problem since $X$ is smooth. and the norm  $\Lambda^n d s_E(y)$ with respect to $\theta$ is bounded uniformly away from 0).

\subsection{Analytic regularisation of psh functions} \label{subsec-regu}

Let $(X, \omega)$ be a compact K\"ahler manifold. Let $L$ be a big and semi-ample line bundle on $X$ (hence $X$ is forced to be projective by Moishezon's theorem). Let $\rho, \Phi_{k_L},X'$ be as in Introduction. From now on  we assume the following hypothesis:\\

\textbf{(H) \quad}  The Chern class of $K^*_X$ contains a closed positive $(1,1)$-current of bounded potentials, \emph{i.e,} there exists a bounded $\eta_\omega$-psh function $\vartheta$ on $X$, where $\eta_\omega$ is the Chern form of the Hermitian metric on $K^*_X$ induced  by $\omega$. Furthermore, $X'$ has only isolated singularities. 
\\

In particular, this assumption is fulfilled if $-c_1(K_X)$ is semi-positive and $n=2$.  
Let 
$$\theta:= \frac{1}{2 k_L}\ddc \log \sum_{j=1}^{d_{k_L}} |s_j|^2$$
  which is smooth closed form in $c_1(L)$.  Hence $\theta$ is the pull-back of the Fubini-Study form in $\C\P^{d_{k_L}-1}$ under $\Phi_{k_L}$.  Let $h_0$ be a smooth Hermitian metric on $L$ with $c_1(L,h_0)= \theta$.

Fix a smooth Riemannian metric on $X$ and let $\B(x,r)$ be the ball of radius $r$ with respect to this metric. Let $r_X>0$ be a constant so that for every $x \in X$  the closure of the  ball $\B(x,r_X)$ is contained in a local chart of $X$ which is biholomorphic to a ball in $\C^n$.

\begin{lemma}\label{le-demailly-prescribed} There exist  constants $C_0 >0$, $r_0>0$  small enough such that for every $y \in X$, there exist global negative $\theta$-psh functions $u_y$ on $X$ so that 
$$u_{y}(x) \le  \log |y-x|+ C_0$$
on $\Phi_{k_L}^{-1}\big(\B(y', r_0)\big)$ where $y':= \Phi_{k_L}(y)$ and by abuse of notation, for every $r>0$, we denote by $\B(y', r)$ the ball of radius $r$ centered at $y' \in \C\P^{d_{k_L}-1}$ with respect to a fixed smooth metric on $\C\P^{d_{k_L}-1}$. Furthermore, for every constant $\epsilon>0$, we have  
\begin{align} \label{ine-uysualaithoi}
u_y \ge \log \epsilon-C
\end{align}
outside $\Phi_{k_L}^{-1}\big(\B(y', \epsilon)\big)$ for some constant $C$ independent of $y,\epsilon$. 
\end{lemma}

\proof
Let $y':= \Phi_{k_L}(y)$ and let $v_y(z)$ be a $\omega_{FS}$-psh function on $\C\P^{d_{k_L}-1}$ given by $$v_y(z):= \log |z-y'|$$ where we use homogeneous coordinates for $z,y'$, and $\omega_{FS}$ is the Fubini-Study form on $\C\P^{d_{k_L}-1}$. Thus, for every $\epsilon>0$, there holds $v_y \ge C \log \epsilon$ outside $\B(y',\epsilon)$ for some constant $C$ independent of $y, \epsilon$.  

Since $\Phi_{k_L}^*\omega_{FS}=\theta$, we infer  $u_y:= \Phi_{k_L}^* v_y$ and $\tilde{u}_y:=\Phi_{k_L}^* \tilde{v}_y$ are $\theta$-psh and satisfies that
$$u_{y}(x) = \log |\Phi_{k_L}(y)- \Phi_{k_L}(x)| \le \log |y-x|+C_0$$
on   $\Phi_{k_L}^{-1}\big(\B(y', r_0)\big)$. Moreover one also has (\ref{ine-uysualaithoi})  because of the smoothness of $v_y$ outside $y'$. 
%The inequality (\ref{ine-uysualaithoi2}) follows from Lojasiewicz's inequality (applied to $N$ and the determinant of the Jacobian of $\Phi_{k_L}$) and (\ref{ine-uysualaithoi}).
\endproof

We denote by $Sing(X')$ the singular locus of $X'$. By the hypothesis, the set $Sing(X')$ is finite. 
Let $\B_r(y)$ be the ball of radius $r$ centered at $y$ in $\C^{k_L-1}$. If $y=0$, then we write $\B_r$ for $\B_r(y)$.  Put $N':= \Phi_{k_L}(N)$ which is an analytic subset in $X'$. %For $z \in \C \P^{k_L-1}$, we denote by $\dist(z, N')$ the distance from $z$ to $N'$ in $\C \P^{k_L-1}$ with respect to the Fubini-Study metric $\omega_{FS}$.
Since $Sing(X')$ is finite, we can find open subsets $U'_1, \ldots, U'_l$  in $\C\P^{d_{k_L}-1}$ such that the following properties hold:

(i) $U'_j \Subset U''_j$ which is biholomorphic to the ball $\B_3$ in $\C^{d_{k_L}-1}$ under a map $\Psi_j$ for every $1 \le j \le l$ and $U_j$ is biholomorphic to $\B_2$ under $\Psi_j$. Furthermore, if $U'_j\cap Sing(X')$ is non empty, then $U'_j \cap Sing(X')$ is contained in $\B_1 \Subset U'_j$,
 % the hyperplane line bundle $\mathcal{O}(1)$ of $\C\P^{d_{k_L}-1}$ is trivial over $U'_j$,  and 

(ii) $X'\subset \cup_{j=1}^l U'_j$,

%(iii) $\overline U'_j$ is contained in some local chart $U''_j$ in $\C\P^{d_{k_L}-1}$,

(iii) There is a hyperplane $H_j$ on $\C\P^{d_{k_L}-1}$ such that $H_j$ does not intersect $U''_j$ for every $1 \le l  \le j$.

By our choice of $U'_j$, we see that $U'_j$ is hyperconvex (hence weakly pseudoconvex), i.e, there is a smooth psh function $w_j$ on $U'_j$ such that $\{w_j < c\}$ is relatively compact in $U'_j$ for every constant $c<0$ and every $j$. Let 
$$U_j:= \Phi_{k_L}^{-1}(U'_j).$$ %, \quad \tilde{U}_j:= \Phi_{k_L}^{-1}\circ \Psi_j^{-1}(\B_1).$$
 Note that $U_j$ is also hyperconvex  and $L$ is trivial over $U_j$ because $L= \Phi_{k_L}^*\mathcal{O}(1)$ and $\mathcal{O}(1)$ is trivial over $X' \backslash H_j$.

\begin{lemma}\label{le-demailly-prescribed0} 
Let $h:= h_0e^{-2 \phi}$ be a singular positively curved metric on $L$ such that $\phi$ is smooth outside $N$. Fix $1 \le j \le l$.  Let $y \in U_j \backslash N$. Let $e$ be a local holomorphic frame of $L$ over $U_j$. Then for every $a \in \C$, there exists a section $f \in H^0(U_j, L)$ such that  $f(y)= a e(y)$ and 
$$\int_{U_j} |f|^2_{h_0} e^{-2 \phi} \theta^n \le C |a|^2 |e(y)|_{h_0}^2 e^{- 2 \phi(y)}.$$
where $C>0$ is a constant independent of $y$ and $a$. 
\end{lemma}

\proof We treat first the case where $\phi$ is smooth. Let $L|_{U_j}$ be the restriction of $L$ to $U_j$, and  $E:= L|_{U_j} \oplus \cdots \oplus L|_{U_j}$ ($n$ times). Since $L$ is trivial on $U_j$, so is $E$. Equip $E$ with the trivial Hermitian metric. Hence a section of $E$ is simply a collection of $n$ holomnorphic functions on $U_j$.  Let $z=(z_1, \ldots, z_{k_L-1})$ be the local coordinates on $U''_j \approx \B_3$. We can assume that $y'$ is the origin in these local coordinates.  Let $s'_E(z):=z$. Observe that $s_E:= s'_E \circ \Phi_{k_L}^{-1}$ is a section of $E$ on $U_j$ and vanishes only at $y$. Recall that $\theta= \Phi_{k_L}^* \omega_{FS}$, where $\omega_{FS}$ is the Fubini-Study form on $\C \P^{k_L-1}$.

Let $N':= \Phi_{k_L}(N)$. Then $X' \backslash N'$ is smooth and $\Phi_{k_L}$ is a biholomoprhism from $U_j \backslash N$ to $U'_j \backslash N'$.   
Let $X'':= X' \cap \B_2$, we have a natural inclusion $\xi: X'' \to U'_j\approx \B_2$. 
Let $\Psi$ be an orthogonal change of coordinates on $\C^{k_L-1}$ so that $\Psi_* \xi_* T_y X''$ is given by  the subspace $\{z_1, \ldots, z_n, 0, \ldots, 0\}$ at $0$ in $\C^{k_L-1}$. Write $\Psi=(\Psi_1, \ldots, \Psi_{k_L})$.
Let 
$$s'_E:=(\Psi_1, \ldots, \Psi_n) \circ \xi$$ regarded as a section of $\Phi_{k_L-1}^* E$.
Let  $Y:=X'' \cap \{s'_E=0, \det J_{s'_E} \not =0: 1 \le k \le n\}$ contains $0$ as an isolated point, where $\xi_y(z):= (z_{j_1}, \ldots, z_{j_n})$ for $z \in X''$ (note that $Y$ may not be connected). 
Note that $\omega_0$ is preserved under $\Psi$. 
By the choice of $s'_E$, there is a constant $\epsilon_0 >0$ independent of $y$ such that 
\begin{align}\label{ine-determinsE}
|\Lambda^n d s'_E(y')|_{\xi^* \omega_{FS}} \ge \epsilon_0.
\end{align}
Indeed, by the choice of $\Psi$,  the norm $|\Lambda^n d s'_E(y')|_{\xi^* \omega_{0}}$ (which is the norm of $\det J_{s'_E}$ with respect to $\xi^*\omega_0$) is equal to the absolute value of the determinant of the $(n,n)$-submatrix of the Jacobian of $(\Psi_1, \ldots, \Psi_n)$ given by the first $n$ rows. Hence  $|\Lambda^n d s'_E(y)|_{\xi^* \omega_{0}}=1$. Since $\omega_{FS}$ and $\omega_0$ are equivalent on $U'$, we get (\ref{ine-determinsE}).     

Let $s_E:= s'_E \circ \Phi_{k_L}$. Observe that the expression $|\Lambda^n ds_E|^2_\omega \omega^n$ is independent of $\omega$ and we have 
$$|\Lambda^n ds_E|^2_\omega\, \omega^n= |\Lambda^n ds_E|^2_\theta \, \theta^n= |\Lambda^n ds'_E|^2_{\xi^* \omega_{FS}} \theta^n \ge \epsilon_0 \theta^n$$
by (\ref{ine-determinsE}). 

Applying Corollary \ref{cor-OKYbang0truocda} to $U_j, s_E,N$ implies that there exists    
a section $f \in H^0(U_j \backslash N, L)$ such that  $f(y)= a e(y)$ and 
$$\int_{U_j} |f|^2_{h_0} e^{-2 \phi} \theta^n \le  C |a|^2 |e(y)|_{h_0}^2 e^{- 2 \phi(y)}$$
where $C>0$ is a constant independent of $y$ and $a$. Since $\Phi_{k_L}$ is biholomorphic on $U_j \backslash N$, we see that $f \circ \Phi_{k_L}^{-1}$ is holomorphic on $U'_j \cap (X' \backslash Sing(X'))$ (which is a normal variety). Hence $f \circ \Phi_{k_L}^{-1}$ extends to be a holomorphic function on $U'_j \cap X'$ (see \cite{Grauert-Remmert-coherentsheaves}). We infer that $f \in H^0(U_j,L)$.

It remains to treat the case where $\phi$ is not necessarily smooth. Put $\phi':= \phi \circ \Phi_{k_L}^{-1}$ which is well-defined on $X' \backslash Sing(X')$. Since $X'$ is normal, $\phi'$ extends to a $\omega_{FS}$-psh function on $X'$. By \cite{Coman-Guedj-Zeriahi} we can extends $\phi'$ to an $\omega_{FS}$-function on $\C \P^{d_{k_L}-1}$. Standard regularisation (see \cite{Blocki-Kolodziej}) gives us a decreasing sequence of smooth $\omega_{FS}$-functions to $\phi'$. Hence we obtain a decreasing sequence of smooth $\theta$-psh functions $(\phi_l)_l$ to $\phi$. Applying the first part to each $\phi_l$ we get $f_l\in H^0(U_j,L)$ with $f_j(y)=a e(y)$ and 
$$\int_{U_j} |f_l|^2_{h_0} e^{-2 \phi_l} \theta^n \le C |a|^2 |e(y)|_{h_0}^2 e^{- 2 \phi_l(y)}.$$
Let $V \Subset U_j \backslash N$ be an open set. We infer that $f_l \in L^2(V,\theta^n)$. Extracting a subsequence if necessary, we can assume that $f_l$ converges uniformly to $f$ in $V$. This combined with the fact that $\phi_l$ is bounded uniformly on $V$ (by the hypothesis on $\phi$) gives 
$$\int_{V} |f|^2_{h_0} e^{-2 \phi} \theta^n \le \liminf_{l \to \infty} \int_{V} |f_l|^2_{h_0} e^{-2 \phi_l} \theta^n \le   C |a|^2 |e(y)|_{h_0}^2 e^{- 2 \phi(y)}.$$
Letting $V$ increasing to $U_j \backslash N$ gives to the desired estimate for $f$ (and again $f\in H^0(U_j,L)$ as in the end of the first part of the proof). The proof is complete.
\endproof

 Since $L$ is big, by Demailly \cite{Demailly_analyticmethod}, there exists a negative $\theta$-psh function $\rho$ such that locally 
$$\rho= \log \big(\sum_{j=1}^r |f_j|\big)+ O(1),$$ for some local holomorphic functions $f_1,\ldots, f_r$, and 
$$\ddc \rho+ \theta \ge \delta_0 \omega,$$
 where $\delta_0>0$ is a constant. We can choose $\rho$ so that $N:= \{\rho = -\infty\}$ is equal to the non-K\"ahler locus of $c_1(L)$, see \cite{Boucksom_anal-ENS}. Recall that the non-K\"ahler locus of $c_1(L)$    is equal to the augmented base locus of $L$ (see \cite[Theorem 2.3]{Tosatti-survey-on-nakamaye} or \cite{Boucksom-these}).
 
Let $h$ be a positive Hermitian metric on $L$. Let $e_L$ is a local holomorphic frame for $L$ (i.e., $e_L$ is a local holomorphic section of $L$ and $e_L \not = 0$ everywhere).  Write $h= h_0 e^{-2 \varphi}$. Thus by hypothesis one gets
$$0 \le c_1(L,h)= - \ddc \log |e_L|_h = \ddc \varphi+ \theta.$$ 
In other words, $\varphi$ is $\theta$-psh function. By multiplying a large constant with $h_0$, without loss of generality we can assume that $\varphi \le 0$. We assume from now on that $\varphi$ is bounded.

 For every constant $\delta\in (0,1)$, define
$$\varphi_\delta:= (1- \delta) \varphi+ \delta \rho, \quad h_\delta:= h_0 e^{-2 \varphi_\delta}.$$
We have $\ddc \varphi_\delta+ \theta \ge \delta \delta_0 \omega$. Let $m \in\N$ and $d_m:= \dim H^0(X, L^m)$ which is $\approx m^n$ as $m \to \infty$. Let $m_0:= 2n+3$.
%where $\lambda$ is the constant in Lemma \ref{le-demailly-prescribed} and $[\cdot]$ is the greatest integer function. %   be a a constant to be chosen later. The value of $c_0$ will depend only on $n$ and $\vol(L)$. 
Let $a_0 \in (0, 1/2)$ be a constant such that 
$$\int_X e^{-2a_0 \rho} \omega^n < \infty.$$
For $m>m_0$, $\delta \in (0, a_0/m)$ and $s,s' \in H^0(X, L^m)$, we put 
$$\langle s, s' \rangle_{L^2}=\langle s, s' \rangle_{L^2,m,\delta}:= \int_X \langle s, s' \rangle_{h_0^m} e^{-2(m-m_0) \varphi_\delta} \omega^n$$
which is finite because the boundedness of $\varphi$ and the choice of $a_0$. To give readers a hint what we do with $\delta$, we remark that we will choose later $\delta:= m^{-2D}$ for some constant $D \ge 1$ (see the proof of Lemma \ref{le-logcontinuity} below), thus the condition $\delta< a_0/m$ is automatically satisfied for $m \ge a_0^{-1}$.
 
Let $\{\sigma_1, \ldots, \sigma_{d_m}\}$ be an orthonormal basis of $H^0(X, L^m)$ with respect to $L^2$-product, and let 
$$\psi_{m,\delta}:= \frac{1}{2m} \log \bigg(\sum_{j=1}^{d_m} |\sigma_j|^2_{h_0^m}\bigg)= \frac{1}{2m}\sup_{s \in H^0(X, L^m): \|s\|_{L^2}=1} \log |s|^2_{h_0^m}.$$
Since $\log |\sigma_j|_{h_0^m}$ is $m \theta$-psh, we infer that $\psi_m$ is $\theta$-psh. %We will need the following important result which is more or less a consequence of \cite{Demailly-numerical-criterion}. 

The following result is a variant from \cite[Theorem 14.21]{Demailly_analyticmethod}. Recall $N=\{\rho = -\infty\}$. 

\begin{theorem}\label{th-analytic-approx} There exists a constant  $C>0$ such that for every $\delta \in (0,a_0/m)$ and every $m \ge m_0+1$
	 there holds:

(i)  
$$\frac{m-m_0}{m}\varphi_\delta (x)- \frac{C+|\log \delta| }{2m} \le \psi_{m,\delta} \le \frac{m-m_0}{m} \sup_{x' \in \B(x, r)} \varphi_\delta(x') + Cr+  C \frac{|\log r|}{m},$$ 
for every $x \in X$ and $r>0$.

(ii) 
$$|\nabla \psi_{m,\delta}(x)| \le 
C+\dfrac{C}{m\delta^{1/2} r^{n+1}}e^{(m-m_0)(\sup_{\B(x,r)} \varphi_\delta-\varphi_{\delta}(x))+ C(m-m_0)r},$$ for every $x \in X \backslash N$ and $r>0$. 
%. Then
\end{theorem}

\proof We check the second inequality in (i). Let $U$ be a small local chart around $x$. We trivialize $L$ over $U$ and let $e_{L,U}$ be a nowhere vanishing holomorhic section of $L$ over $U$. %We normalize $e_{L,U}$ so that $\|e_{L,U}\|_{h_0}=1$.
Hence we can identify $h_0= e^{-2 \phi_0}$ for some smooth function $\phi_0$, and sections of $L^m$ are identified  with holomorphic functions on $U$.

Let $s \in L^m$ with $\|s\|_{L^2}=1$. By abuse of notation we still denote by $s$ the holomorphic function corresponding to a section $s$ of $L^m$. Thus $|s|^2_{h_0^m}= |s|^2 e^{- 2m \phi_0}$. By the submean inequality, one gets
$$|s(x)|^2 \lesssim r^{-2n} \int_{\B(x,r)}|s|^2 d \Leb_{\C^n} \lesssim r^{-2n} e^{ \sup_{\B(x,r)} q} \int_{\B(x,r)} |s|^2 e^{-q} \omega^n
 \lesssim r^{-2n} e^{ \sup_{\B(x,r)} q} \|s\|^2_{L^2},$$
where $q=2(m-m_0)\varphi_{\delta}+2m \phi_0$.
Then
$$|s(x)|^2_{h_0^m}= |s(x)|^2 e^{-2 m \phi_0(x)} \lesssim r^{-2n} e^{2(m-m_0) (\sup_{\B(x,r)} (\varphi_\delta+ \phi_0)- \phi_0(x))}.$$
By this and the fact that $\phi_0 \in \mathcal{C}^1$ we infer that 
$$|s(x)|^2_{h_0^m}\lesssim r^{-2n} e^{2(m-m_0) \sup_{\B(x,r)} \varphi_\delta+2(m-m_0)C_1r},$$
for some constant $C_1>0$ independent of $\delta, m, \varphi,s$, for $s \in H^0(X,L^m)$ with $|s|_{L^2}= 1$.
It follows that
\begin{equation}\label{eq chantrenpsim}
e^{2m\psi_{m, \delta}}=\sup_{s \in H^0(X, L^m): \|s\|_{L^2}=1}|s(x)|^2_{h_0^m}\leq e^{C_2}r^{-2n} e^{2(m-m_0) \sup_{\B(x,r)} \varphi_\delta+2(m-m_0)C_1r},
\end{equation}
 where $C_1, C_2>0$ are constants independent of $\delta, m, \varphi,s.$
 Hence we obtain 
$$\psi_{m,\delta} \le \frac{m-c_0}{m} \sup_{x' \in \B(x, r)} \varphi_\delta(x') + C_1r+  \frac{2n|\log r|+C_2}{2m}
\leq \frac{m-m_0}{m} \sup_{x' \in \B(x, r)} \varphi_\delta(x') + C_3r+  \frac{C_3|\log r|}{m},$$ 
for every $x \in X$ and $r>0$, where $C_3=C_1+C_2e+n$.

The remaining inequality of (i) requires the $L^2$-estimate. It suffices to consider $x \not \in N$.  Let $U_1,\ldots, U_l$ be the open cover of $X$ defined above. Without loss of generality we can assume that $x\in U_1$. Choose $U:= U_1$.  We can modify the coordinates on  $U'_1 \approx \B_2$ so that $\Phi_{k_L}(x)$ is the center of $U'_1$.  By Lemma \ref{le-demailly-prescribed0}, there are a constant $B_1>0$ independent of $x$ such that for every $a \in \C$,  there is a $f \in H^0(U_1, L^m)$ so that $f(x)= a e_{L,U}^m$ and 
$$\int_{U_1} |f|_{h_0^m}^2 e^{-2 (m-m_0)\varphi_\delta} \theta^n \le B_1 |a|^2 |e_{L,U}(x)|^{2m}_{h_0} e^{-2 (m-m_0)\varphi_\delta(x)}.$$ 

Fix a cut-off function $\chi'$ supported on $U'_1$ and equal to $1$ on $\B_{1}$ (recall that $U'_1 \approx \B_2$ and $\Phi_{k_L}(x)=0$ is the origin). By our choice of $U'_j$, we see that $Sing(X') \cap U'_j$ is contained in the interior of $\{\chi'=1\}$.

Put $\chi:= \chi' \circ \Phi_{k_L}^{-1}$. By Lemma \ref{le-demailly-prescribed}, there exist a constant $C_4>0$ independent of $x$ and  a  negative $\theta$-psh function $u_x$ satisfying $u_x(z) \le \log |z-x|+ C_4$,  and $u_x(z) \ge -C_4$ for $z \not \in \Phi_{k_L}^{-1}(\B_{1/2}(x'))$ ($x':= \Phi_{k_L}(x)$).   Let 
$$w:= \big(m- m_0\big)\varphi_\delta+ m_0 u_x.$$
 Observe that 
$$\ddc w+ m \theta \ge \big(m- m_0\big)\delta \delta_0 \omega.$$
 Let 
 $$h_{u_x}:= h_0 e^{-2 u_x}, \quad \tilde{h}: = h_0^m e^{-2 w}$$
which is a singular Hermitian metric on $L^m$. Thanks to the hypothesis about the semi-positivity of  $K^*_X$,   we  can apply Corollary \ref{cor-dbar} to $\tilde{h}$ and $g:= \bar \partial (\chi f)$ which is smooth. Hence we find  a smooth section $v$ of $L$ over $X$ so that $\bar \partial v= g $ and
\begin{align}\label{ine-giaidbarptcu}
\int_X |v|^2_{h_0^m} e^{-2w} \omega^n \le  \frac{1}{(m- m_0)\delta\delta_0} \int_{U_1 \backslash \Phi_{k_L}^{-1}(\B_{1}(x'))} |f|^2_{h_0^m} e^{-2 w} \omega^n
\end{align}
because $\partial \chi'$ vanishes on $B_1(x')$. This combined with the fact that $u_x(x') \ge C_4$ outside $\Phi_{k_L}^{-1}(\B_{1}(x'))$ yields 
\begin{align}\label{ine-giaidbarpt}
\int_X |v|^2_{h_0^m} e^{-2w} \omega^n &\le  \frac{1}{(m- m_0)\delta\delta_0} \int_{U_1\backslash \Phi_{k_L}^{-1}(\B_{1}(x'))} |f|^2_{h_0^m} e^{-2(m-m_0)\varphi_\delta} \omega^n\\
\nonumber
&\le  \frac{1}{(m- m_0)\delta\delta_0} \int_{U_1\backslash \Phi_{k_L}^{-1}(\B_{1}(x'))} |f|^2_{h_0^m} e^{-2(m-m_0)\varphi_\delta} \theta^n\\
\nonumber
& (\text{because $\theta^n$ and $\omega^n$ are comparable on $U_1\backslash \Phi_{k_L}^{-1}(\B_{1}(x')$})\\
\nonumber
& \le |a|^2 |e_{L,U}(x)|_{h_0}^{2m} e^{-2(m-m_0)\varphi_\delta(x)}.
\end{align}
Note that since $g$ vanishes near $x$, one gets that $\bar\partial v =0$ near $x$. Thus $v$ is holomorphic near $x$.  By properties of $u_{x}$, observe that  
$$e^{-2w(x')} \gtrsim \frac{1}{|x'-x|^{2n+2}}$$
which in turn implies that $v(x)=0$ because $\int_X |v|^2_{h_0^m} e^{-2 w} \theta^n$ is finite and $\theta$ is K\"ahler near $x$ (recall that $x \not \in N$). 
 This together with (\ref{ine-giaidbarpt}) gives
\begin{align}\label{ine-giaidbarpt2}
\int_X |v|^2_{h_0^m} e^{-2w} \omega^n \lesssim  \frac{|a|^2 }{(m- m_0)\delta} e^{-2(m-m_0)\varphi_\delta(x)-2m \phi_0(x)}.
\end{align}
Let $\tilde{v}:= \chi f - v \in H^0(X, L^m)$. %The function $\tilde{v}$ extends to a global holomorphic section of $L$ on $X$ because $\tilde{v}\circ \Phi_{k_L}^{-1}$ is holomorphic on $X' \backslash N'$, $X'$ is normal and $N'$ is of codimension 2 in $X'$. 
  Since $u_x \le 0$, using  (\ref{ine-giaidbarpt2}) and the choice of $f$, we obtain 
$$\int_X |\tilde{v}|^2_{h_0^m} e^{-2(m- m_0)\varphi_\delta} \omega^n \le  \frac{B_2|a|^2 |e_{L,U}(x)|_{h_0}^{2m}}{(m- m_0)\delta}  e^{-2(m-m_0)\varphi_\delta(x)},$$
for some constant $B_2>0$ independent of $x, a,m, \delta$. 
Choose
$$ \quad a:= B_2^{-1/2} |e_{L,U}(x)|_{h_0}^{-m}\delta^{1/2}(m-m_0)^{1/2} e^{(m-m_0)\varphi_\delta(x)}.$$
We see that 
$$\int_X |\tilde{v}|^2_{h_0^m} e^{-2(m- m_0)\varphi_\delta} \omega^n \le 1.$$
and 
$$\tilde{v}(x)= f(x)-v(x)= f(x)= a \, e^m_{L,U}(x).$$
It follows that
\begin{equation}\label{eq chanduoipsim}
e^{2m\psi_{m, \delta}(x)}\geq |\tilde{v}(x)|_{h_0^m}^2\geq   \dfrac{\delta(m-m_0)}{B_2}e^{2(m-m_0)\varphi_{\delta}(x)}.
\end{equation}
Thus
\begin{align*}
\psi_{m, \delta}(x) &\ge \dfrac{m-m_0}{m}\varphi_\delta(x)+ \dfrac{1}{2m}\log \dfrac{\delta(m-m_0)}{B_2} \\ 
&\geq \dfrac{m-m_0}{m}\varphi_\delta(x)-\dfrac{|\log\delta|+\log B_2}{2m} \cdot
\end{align*}
This finishes the proof for (i).

We now check (ii).  We work in a small local chart $(U, x)$ and write $h_0= e^{- 2 \phi_0}$ as above. We identify sections with holomorphic functions on a trivialization of $L$ over this local chart. We have 
$$\psi_{m,\delta}= \frac{1}{2m} \log \sum_{j=1}^{d_m} |\sigma_j|^2 -\phi_0.$$
Direct computations give
\begin{align*}
\partial \psi_{m,\delta}=  \frac{1}{2m} \frac{\sum_{j=1}^{d_k} \bar \sigma_j \partial\sigma_j }{ \sum_{j=1}^{d_k} |\sigma_j|^2} -\partial \phi_0.
\end{align*}
Hence
\begin{align}\label{eqdanhgianablapsim0}
|\partial \psi_{m,\delta}| \le   \frac{1}{2m} \frac{\big(\sum_{j=1}^{d_k}  |\partial\sigma_j|^2\big)^{1/2} }{ \big(\sum_{j=1}^{d_k} |\sigma_j|^2\big)^{1/2}} + |\partial \phi_0|.
\end{align}
By \eqref{eq chanduoipsim}, we have 
\begin{equation}\label{eqdanhgianablapsim1}
\sum_{j=1}^{d_k} |\sigma_j(x)|^2 = e^{2m(\psi_{m, \delta}+\phi_0)}
\ge \dfrac{\delta(m-m_0)}{B_2}e^{2(m-m_0)\varphi_{\delta}(x)+2m\phi_0(x)}.
\end{equation}
 On the other hand, since $\sigma_j$ is homomorphic, it follows from Cauchy's integral formula that
 
 $$\sum_{j=1}^{d_k}  |\partial\sigma_j(x)|^2\lesssim r^{-n-2}\sum_{j=1}^{d_k} \int_{x+\partial_0\Delta_r^n}|\sigma_j|^2d\xi_1...d\xi_n
 \lesssim r^{-2}\sup_{x+\Delta_r^n}\sum_{j=1}^{d_k}|\sigma_j|^2=r^{-2}\sup_{x+\Delta_r^n}e^{2m(\psi_{m, \delta}+\phi_0)},$$
 for every $0<r<\dist (x, \partial U)$, where $\Delta_r$ denotes the disk of radius $r$ with center at $0$ in $\C$, and $\partial_0 \Delta_r^n:= (\partial \Delta_r)^n$.
 Therefore, by \eqref{eq chantrenpsim}  and  the fact $\phi_0\in \mathcal{C}^1$, we get
 \begin{equation}\label{eqdanhgianablapsim2}
 \sum_{j=1}^{d_k}  |\partial\sigma_j(x)|^2\lesssim r^{-2n-2}e^{2(m-m_0) (\sup_{\B(x,r)} \varphi_\delta+\phi_0)+C_5(m-m_0)r}.
 \end{equation}
Combining \eqref{eqdanhgianablapsim0}, \eqref{eqdanhgianablapsim1} and \eqref{eqdanhgianablapsim2}, we get 

$$|\partial \psi_{m,\delta}| \lesssim 1+\dfrac{1}{m\delta^{1/2} r^{n+1}}e^{(m-m_0)(\sup_{\B(x,r)} \varphi_\delta-\varphi_{\delta}(x))+ C_5(m-m_0)r}.$$
\endproof

The new point here is that we approximate $\varphi$ through the analytic approximation sequence for $\varphi_\delta$ with $\delta$ depending on $m$. We will choose $\delta$ to be very small compared to $m$. % in order to obtain a crucial  upper bound for the gradients of approximation functions

\begin{lemma} \label{le-hieugiuasupvaham} Let $u$ be a bounded negative psh function on the open unit ball $\B$ of $\C^n$ 
	and let $K\Subset \B$. Let $v$ be a H\"older continuous plurisubharmonic function on $\B$ and denote $\mu=(dd^c v)^n$.
	 Then there exist constants $\alpha=\alpha (\mu, K)$ and $C=C(n, K,  \mu, \|u\|_{L^{\infty}})>0$ such that
$$\int_K\big|\sup_{x'\in\B(x, s)}u(x')-u(x)\big|d\mu \leq C s^{\alpha},$$
for every $0<s<r_0^3$, where $r_0:=\frac{1}{4}\inf_{w\in K}\dist(w, \partial\B)$.
\end{lemma}

\proof 
Denote $M=\|u\|_{L^{\infty}}$, $U=(1-3r_0)\B$ and $V=(1-2r_0)\B$. First, we prove that 
\begin{equation}\label{eq0.1lehieugiuasupvaham}
\int_{V}\big|\sup_{y \in \B(x,s)} u(y) - u(x)\big| d\Leb \le C_0M s^{2/3},
\end{equation}
for every $0<s<r_0^3$, where $C_0>0$ is a constant depending only on $n$ and $r_0$.

For every $0<r<r_0$ and $z\in U$, we denote
$$\hat{u}_r(z)=\frac{1}{\vol (\B(z, r))}\int_{\B(z, r)}u(\xi)dV(\xi),$$
and 
$$\bar{u}_r(z)=\sup_{\xi\in\B(z, r)}u(\xi).$$
Let $z_0 \in V$ and $v_M:= \frac{M}{r_0}(|z-z_0|^2-1)$. We have $v_M < u$ on $\B_{\sqrt{1/2- r_0}}(z_0)$ (which contains $V+ r_0 \B$ because $r_0 < 1$).  By the comparison principle for Laplace operator, one has 
$$\int_{\{v_M < u\}} \Delta u \le \int_{\{v_M < u\}} \Delta v_M \lesssim M/r_0.$$
It follows that  there exists $C_1>0$ depending only on $n$ such that
$$\int_{V+r_0\B}\Delta u\leq \frac{C_1 M}{r_0}.$$
Then, by Jensen formula (see, for example, \cite{Bedford_Taylor_76, DemaillyHiep_etal}), one has
\begin{equation}\label{eq1lehieugiuasupvaham}
	\int_{V}|\hat{u}_r(z)-u(z)|d\Leb \leq C_2M r^2,
\end{equation}
for every $0<r<r_0$, where $C_2>0$ depends only on $n$ and $r_0$. 

For every $z\in V$ and for every $0<s<r$, there exists $\hat{z}\in\overline{\B(z, s)}$ such that
$$\bar{u}_s(z)=u(\hat{z})\leq\hat{u}_r(\hat{z}).$$
Since $u$ is negative, it follows that
\begin{equation}\label{eq2lehieugiuasupvaham}
	\bar{u}_s(z)\leq\left(\frac{r-s}{r}\right)^{2n}\hat{u}_{r-s}(z).
\end{equation}
Combining \eqref{eq1lehieugiuasupvaham} and \eqref{eq2lehieugiuasupvaham}, we get
\begin{align*}
	\int_V|\bar{u}_s(z)-u(z)|\d\Leb&\leq \left(\frac{r-s}{r}\right)^{2n}\int_V|\hat{u}_{r-s}(z)-u(z)|d\Leb
	+\frac{r^{2n}-(r-s)^{2n}}{r^{2n}}\int_V|u(z)|d\Leb\\
	&\leq C_3\left(M(r-s)^2+\frac{M s}{r}\right),
\end{align*}
for every $0<s<r<r_0$, where $C_3>0$ is a constant depending only on $n$ and $r_0$. Choosing $s=r^3$, we obtain \eqref{eq0.1lehieugiuasupvaham} (with $C_0=2C_3$).

Fix $s\in (0, r_0^3)$. For every $z\in V$, we denote $\psi(z)=\frac{M}{r_0}(|z|^2-(1-2r_0)^2)$, $u':= \max\{u, \psi\}$ and $v':= \max\{\bar{u}_s,  \psi\}$. We have $u=u'$ on $K$, $\bar{u}_s=v'$ on $K$ and $u'=v'=\psi$ on $V\setminus U$.
 Let $\phi\in \Cc_0^{\infty}(V)$ such that $0\leq\phi\leq 1$ and $\phi\equiv 1$ on $U$. Put $\tilde{u}=\phi u'$ and $\tilde{v}=\phi v'$. By using the standard embedding $\C^n\hookrightarrow\C\P^n$, one can extend $\tilde{u}$
 and $\tilde{v}$ to $A\omega_{FS}$-plurisubharmonic functions on $\C\P^n$, where $A\geq 1$ is a constant depending
 only on $n, M$ and $r_0$. Since $\mu=(dd^cv)^n$, we have $\tilde{\mu}:=\bf{1}_{V}\mu$ is a H\"older continuous measure on $\C\P^n$. Therefore, there exist constants $\beta=\beta (\tilde{\mu})>0$ and
  $C_4=C_4(\tilde{\mu}, M, A)>0$ such that
  \begin{equation}\label{eq0.2lehieugiuasupvaham}
 \int_K|\bar{u}_s-u|d\mu\leq \|\tilde{u}-\tilde{v}\|_{L^1(\tilde{\mu})}\leq C_4\|\tilde{u}-\tilde{v}\|_{L^1(\C\P^n)}^{\beta}
  \leq C_4\|\bar{u}_s-u\|_{L^1(V)}^{\beta}.
  \end{equation}
 Combining \eqref{eq0.1lehieugiuasupvaham} and \eqref{eq0.2lehieugiuasupvaham}, we get
 $$\int_K|\bar{u}_s-u|d\mu \leq C_5 s^{2\beta/3},$$
 where $C_5>0$ is a constant depending on $n, r_0, \tilde{\mu}$ and $M$.
The proof is completed.
\endproof

Recall that  $N= \{\rho = -\infty\}$. By the choice of $\rho$, and Lojasiewicz's inequality (e.g., see \cite{Bierstone_Milman}), there exist constants $A_0, A_1>1$ such that  
\begin{align}\label{ine-chanduoicuarho}
 A_0 \log \dist(x, N)-A_1 \le \rho(x) \le  \frac{1}{A_0} \log \dist(x, N) +A_1,
\end{align}
for every $x \in X$.

\begin{theorem}\label{th-approxthat} Let $\mu$ be a H\"older continuous measure on $X$ and $p \ge 1$ be a constant.  Assume that $\varphi$ is bounded on $X$ and $B:= \|\varphi\|_{L^\infty}$. Then there exist a constant $C>0$  and  a family of $\theta$-psh functions $\psi_{m,\delta}$ with $\delta \in (0, a_0/m), m  \in \Z^+$ satisfying the following three properties:

(i) $$\|\psi_{m,\delta}- \varphi\|_{L^p(\mu)} \le C\frac{|\log \delta|+ \log m}{m}+ C\delta,$$

(ii) $$\psi_{m,\delta}(x) \ge \varphi(x)- \frac{B m_0}{m} + A_0(\delta+m^{-1}) \log \dist(x, N) -  
C \left(\delta+\frac{|\log \delta|}{m}\right)$$
for every $ x \in X$, 

(iii) %for every $\tau \in [0,1]$, there is a constant $C_\tau$ independent of $\delta,m,x$ so that 
 $$|\nabla \psi_{m,\delta}(x)| \le C \delta^{-1/2} e^{(B+1)m} e^{-A_0 m \delta \log \dist(x,N)}$$ for every $x \in X$. 
\end{theorem}

\proof We note that the assumption that $\varphi$ is bounded implies that the Chern class of $L$ is nef by Demailly's regularisation theorems.  The property  $(ii)$ follows from  Theorem \ref{th-analytic-approx} $(i)$ and
from \eqref{ine-chanduoicuarho}. The property $(iii)$ follows from Theorem \ref{th-analytic-approx} $(ii)$
 applied to $r=1$ and from (\ref{ine-chanduoicuarho}). It remains to prove $(i)$.
 
 Since $\varphi_{\delta} =(1-\delta)\varphi +\delta \rho$, we have
 \begin{equation}\label{eq1thapprothat}
 \sup_{\B(x, r)}\varphi_{\delta}-\varphi_{\delta}(x)\leq (1-\delta)(\sup_{\B(x, r)}\varphi-\varphi(x))+\delta |\rho(x)|
 \leq \sup_{\B(x, r)}\varphi-\varphi(x)+\delta |\rho(x)|,
 \end{equation}
 and
 \begin{equation}\label{eq2thapprothat}
 \left|\frac{m-m_0}{m}\varphi_{\delta}(x)-\varphi(x)\right|\leq \left(\frac{m_0}{m}+\delta\right)|\varphi(x)|+\delta|\rho(x)|,
 \end{equation}
 for every $x\in X$, $m>m_0$ and $0<\delta<1$.
 
 Using \eqref{eq1thapprothat}, \eqref{eq2thapprothat} and Theorem \ref{th-analytic-approx} $(i)$, we get
 \begin{align*}
 |\psi_{m,\delta}- \varphi|&
 \leq \left|\psi_{m,\delta}- \frac{m-m_0}{m}\varphi_{\delta}\right|+  \left|\frac{m-m_0}{m}\varphi_{\delta}-\varphi\right|\\
 &\leq \sup_{B(x, r)}\varphi_{\delta}- \varphi_{\delta}(x)+C_1r+C_1\frac{|\log r|+|\log\delta|+1}{m}+\left(\frac{m_0}{m}+\delta\right)|\varphi(x)|+\delta|\rho(x)|\\
 &\leq \sup_{\B(x, r)}\varphi-\varphi(x)+2\delta |\rho(x)| +C_1r+B\delta+C_2\frac{|\log r|+|\log\delta|}{m},
 \end{align*}
 for every $m>m_0, r>0$ and $0<\delta<1/2$, where $C_1, C_2>0$ are constants.  Then we have
 \begin{equation}\label{eq3thapprothat}
 \|\psi_{m,\delta}- \varphi\|_{L^p(\mu)}\leq \|\sup_{\B(x, r)}\varphi-\varphi(x)\|_{L^p(\mu)}
 +C_3\left(\frac{|\log \delta|+ |\log r|}{m}+ \delta+r+\delta\|\rho\|_{L^p(\mu)}\right).
 \end{equation}
 
 It follows from \cite[Proposition 4.4]{DinhVietanhMongeampere} that there exist constants $\epsilon, M>0$ depending only on $X, \omega, \theta$ and $\mu$ satisfying
 $$\int_Xe^{-\epsilon w} d\mu\leq M,$$
 for every $w\in\PSH (X, \theta)$ with $\sup_X w=0$. Then, by H\"older inequality, we have
 $$	\|\sup_{\B(x, r)}\varphi-\varphi(x)\|_{L^p(\mu)} \le 
 	\|\sup_{\B(x, r)}\varphi-\varphi(x)\|_{L^1(\mu)}^{\frac{1}{2p}}
 	\|\sup_{\B(x, r)}\varphi-\varphi(x)\|_{L^{2p-1}(\mu)}^{\frac{2p-1}{2p}} \le
 	 C_4 \|\sup_{\B(x, r)}\varphi-\varphi(x)\|_{L^1(\mu)}^{\frac{1}{2p}},$$
 	 where $C_4>0$ is a constant depending only on $M, \epsilon, \mu$ and $p$.
 	  This combined with Lemma \ref{le-hieugiuasupvaham} gives
 	 \begin{equation}\label{eq1th-approthat}
 	 	\|\sup_{\B(x, r)}\varphi-\varphi(x)\|_{L^p(\mu)} \le  C_5 r^{\alpha/p},
 	 \end{equation}
  for every $0<r<r_0$, where $r_0=r_0(X, \omega), \alpha=\alpha (X, \omega, \mu)$ and $C_5=C_5(n, X, \omega, \theta, \mu, B, p)$
  are positive constants.
  
  Combining \eqref{eq3thapprothat} and \eqref{eq1th-approthat}, we get
   $$\|\psi_{m,\delta}- \varphi\|_{L^p(\mu)}\leq
   C_6\left(r^{\alpha/p}+\frac{|\log \delta|+ |\log r|}{m}+ \delta+r\right),$$
  for every $m>m_0$, $0<r<r_0$ and $0<\delta<1/2$. Choosing $r=\frac{r_0}{m^{p/\alpha}}$, we obtain $(i)$.
The proof is completed.
\endproof

\section{Going up to the desingularisation of $N$}\label{sec6}

In this section we will prove Theorem \ref{the-log-continuity}.
Let $\pi: \widehat X \to X$ be the composition of sequence of blowups along smooth centers over $N$ such that $\widehat N:= \pi^{-1}(N)$ is a simple normal crossing hypersurface in $X'$. By Lojasiewicz's inequality, one has 
\begin{equation}\label{eqLoja2}
\dist(\pi(x), N) \lesssim \dist(x, \widehat N) \lesssim \dist^\beta(\pi(x), N),
\end{equation}
for some constant $\beta>0$ independent of $x \in \widehat X$.  Let $\widehat\varphi:= \pi^* \varphi$ which is $\widehat \theta$-psh, where $\widehat \theta:= \pi^* \theta$.

\begin{theorem}\label{th-approxthat-desing} Let $\mu$ be a H\"older continuous measure on $\widehat X$ and $p \ge 1$ be a constant. Assume that $\varphi$ is bounded and let $B:= \|\varphi\|_{L^\infty}$. Then there exist constants $A, C>0$ and  a family of $\widehat \theta$-psh functions $\widehat \psi_{m,\delta}$ with $\delta \in (0, a_0/m), m  \in \Z^+$ satisfying the following three properties:

(i)  $$\|\widehat \psi_{m,\delta}- \widehat \varphi\|_{L^p(\mu)}\le C\frac{|\log \delta|+ \log m}{m}+ C\delta,$$

(ii) $$\widehat \psi_{m,\delta}(x) \ge \widehat \varphi(x)- \frac{B m_0}{m} + A\delta \log \dist(x, \widehat N) -  
C \left(\delta+\frac{|\log \delta|}{m}\right),$$
for every $ x \in \widehat{X}$, 

(iii) $$|\nabla \widehat \psi_{m,\delta}(x)| \le C \delta^{-1/2} e^{(B+1)m} e^{-A m \delta \log \dist(x, \widehat N)},$$
 for every $x \in \widehat{X}$. 
\end{theorem}

\proof Let $\psi_{m,\delta}$ be functions in Theorem \ref{th-approxthat}. Let $\widehat \psi_{m,\delta}:= \pi^* \psi_{m,\delta}$. 
The desired assertions (ii) and (iii) follow directly from \eqref{eqLoja2} and  Theorem \ref{th-approxthat}. To see why (i) holds, we recall that $\pi$ is a composition of successive blowups along smooth centers. Thus the desired inequality (i) is deduced by Theorem \ref{th-approxthat} and Corollary \ref{cor-pushdownMAolder} applied to $\mu$.  The proof is complete. 
\endproof

\begin{lemma}\label{le-logcontinuity} Assume that $\varphi$ is bounded on $X$ and $(\ddc \varphi+ \theta)^n= \mu$ is a H\"older continuous measure. Let $\gamma$ be an arbitrary constant in $(0,1)$. Then for every constant $D >1$, there is a constant  $c_{D, \gamma}>0$ so that
$$|\widehat\varphi(x)- \widehat\varphi(y)| \le \frac{c_{D, \gamma}}{ |\log\dist (x, y)|^\gamma},$$ 
for every $x, y \in \widehat X \backslash \widehat N$ with 
$$\left(\dist (x, y)\right)^D \le \min \{ \dist(x,\widehat N), \dist (y,\widehat N)\}.$$
\end{lemma}

\proof Without loss of generality, we can assume that $0<\dist (x, y)<1/2$. Let $p>1$ be a constant. Denote
 $\gamma_0:= p/(p+2n+1)$ and $\gamma=p/(p+2n+2)$. Note that if $p \to \infty$, then $\gamma \to 1$. Let $\delta:= m^{-2D}$.  By Lemma \ref{lem-chanduoidayulambda} (we choose the constant $\gamma=1$ in Lemma \ref{lem-chanduoidayulambda}) and Theorem \ref{th-approxthat-desing}(i), one get
$$\widehat \psi_{m,\delta}(x)- \widehat\varphi(x) \lesssim_{\gamma_0} \|\widehat \psi_{m,\delta}- \widehat \varphi\|^{\gamma_0}_{L^p(\mu)} \lesssim_{\gamma_0} \left(\frac{\log m}{m}\right)^{\gamma_0},$$
for every $x \in \widehat X$,  $m>m_0$.
This combined with Theorem \ref{th-approxthat-desing} (ii) yields
\begin{align}\label{ine-chanpsimvarphi}
|\widehat\psi_{m,\delta}(x) - \widehat \varphi(x)| \lesssim \left(\frac{\log m}{m}\right)^{\gamma_0}+ m^{-2D} (-\log \dist(x,\widehat N))_+
\end{align}
for every $x \in \widehat X$,  $m>m_0$. Here $(-\log \dist(x,\widehat N))_+=\max\{-\log \dist(x,\widehat N), 0\}$.

 Let $l_{x,y}$ be the curve chosen as in Lemma \ref{le-chonduongcong} (for $\widehat N$ in place of $N$). Now using (\ref{ine-chanpsimvarphi}) and Theorem \ref{th-approxthat-desing} (iii), and Lemma \ref{le-chonduongcong},  we estimate
\begin{align*}
|\widehat\varphi(x)-\widehat\varphi(y)| &\le |\widehat\varphi(x)- \widehat\psi_{m,\delta}(x)|+ |\widehat\varphi(y)- \widehat\psi_{m,\delta}(y)|+ |\widehat\psi_{m,\delta}(x)- \widehat\psi_{m,\delta}(y)| \\
& \lesssim
 \left(\frac{\log m}{m}\right)^{\gamma_0}+  m^{-2D} (-\log \dist(x,\widehat N))_+
 + m^{-2D} (-\log \dist(y,\widehat N))_+\\
& +   \dist (x, y) m^{D} e^{m(B+1)}  e^{-A m^{-2D+1} \log \dist(l_{x,y}(t), \widehat N)},
\end{align*}
for some point $t \in [0,1]$. Since $\dist(l_{x,y}(t),\widehat N) \ge C^{-1} \min \{\dist(x, \widehat N), \dist(y,\widehat N)\}$, we obtain 
\begin{align*}
|\widehat\varphi(x)-\widehat\varphi(y)|  & \lesssim 
\left(\frac{|\log \delta|}{m}\right)^{\gamma_0}+ \delta^{\gamma_0}+  \delta(-\log \min\{ \dist(x,\widehat N), \dist(y,\widehat N)\})_+\\
& + \dist (x, y) \delta^{-1/2} e^{m(B+1)}   e^{-A m^{-2D+1} \log \min \{\dist(x,\widehat N), \dist(y,\widehat N)\}}.
\end{align*}
%Choosing $\delta=m^{-2D}$, we get
%\begin{align*}
%|\widehat \varphi(x)-\widehat \varphi(y)|  & \lesssim 
%m^{-\gamma}+ m^{-2D} \left(- \log \min\{\dist(x,\widehat N), \dist(y,\widehat N)\}\right)_+\\
%& + \dist (x, y) m^{D} e^{m(B+1)}   e^{-A m^{-2D+1} \log \min \{\dist(x,\widehat N), \dist(y,\widehat N)\}}.
%\end{align*}
Hence, if $\left(\dist (x, y)\right)^D \le \min \{ \dist(x,\widehat N), \dist (y,\widehat N)\}$ then we have
\begin{align*}
|\widehat\varphi(x)-\widehat\varphi(y)|  & \lesssim 
m^{-\gamma}-D m^{-2D} \log\dist(x, y)\\
& +  \dist (x, y) m^{D} e^{m(B+1)}   e^{-A D m^{-2D+1} \log\dist (x, y)}.
\end{align*}
By choosing 
$$m:=\max\left\{m_0+1, \frac{\gamma |\log\dist (x, y)|}{3(B+1)}\right\},$$
we get
$$|\widehat\varphi(x)- \widehat\varphi(y)| \le \frac{c_D}{ |\log\dist (x, y)|^{\gamma}},$$ 
for every $x, y \in \widehat X \backslash \widehat N$ with 
$$\left(\dist (x, y)\right)^D \le \min \{ \dist(x,\widehat N), \dist (y,\widehat N)\}.$$
 This finishes the proof. 
\endproof

\begin{proposition} \label{pro-logalphabecontiu} Assume that $\varphi$ is bounded on $X$ and $\theta_\varphi^n$ is a H\"older continuous Monge-Amp\`ere measure. Then for every constant $\gamma \in (0,1)$, there exists a constant $C_\gamma>0$ such that
	\begin{equation*}\label{eqprologalpha}
	|\varphi(x)- \varphi(y)| \le \frac{C_\gamma}{|\log\dist(x, y)|^\gamma},
	\end{equation*}
	for every $x,y \in  X \backslash  N$.
\end{proposition}

\proof
By Lemma \ref{le-logcontinuity}, we can apply Proposition \ref{pro-chuyendudistNdenlogcon} to $\widehat \varphi$, and we see that $\widehat \varphi$ is $\log^{\gamma}$-continuous on $\widehat X$. This combined with Lemma \ref{le-successiveblowup} yields that $\varphi$ is $\log^{\gamma}$-continuous. 
\endproof

%We will show that we can strengthen the last result for any positive constant $\alpha$. 

%\begin{theorem} \label{th-loginfitnycont} Let $X, L, \varphi$ be as above. Assume that $\varphi$ is bounded on $X$ and $\theta_\varphi$ is a H\"older continuous Monge-Amp\`ere measure Then for every $A \in \R_+$ there exists a constant $C_A>0$ so that $$|\varphi(x)- \varphi(y)| \le \frac{C_A}{|\log |x-y||^A}\cdot$$\end{theorem}

%Theorem \ref{the-log-continuity} is a direct consequence of the following result.

%\begin{theorem}\label{the-log-continuityholdermeasure} Let $(X, \omega)$ be a compact K\"ahler manifold such that the Chern class of $-K_X$ contains closed positive current of bounded potentials. Let $L$ be a big and semi-ample line bundle on $X$. Let $\theta$ be a smooth semi-positive form in $c_1(L)$. Let $\mu$ be a H\"older continuous measure on $X$ of mass equal to $\int_X \theta^n$. Then the unique solution $u$  to the equation $(\ddc u+ \theta)^n =\mu$ is $\log^M$-continuous for every constant $M>0$. \end{theorem}

\begin{proof}[Proof of Theorem  \ref{the-log-continuity}]  
	Throughout this proof, $C_j$ ($j=1,2,3...$) is a constant independent of $m,\delta, x, r$.
	
	Let $\gamma\in (0, 1)$. By Proposition  \ref{pro-logalphabecontiu}, we have
	\begin{align*}
\sup_{x'\in \B(x,r)}\varphi_\delta(x')-\varphi_\delta(x) & \le (1-\delta)(\sup_{x'\in\B(x,r)}\varphi(x')-\varphi(x)) - \delta \rho(x)\\
	 &\le C_1\left(|\log r|^{-\gamma}+ \delta |\log \dist(x,N)|\right),
	\end{align*}
for every $x\in X$ and $0<r, \delta<1/2$.
This combined with Theorem \ref{th-analytic-approx} (ii) yields
$$|\nabla \psi_{m,\delta}(x)| \le C_2 \delta^{-1/2}r^{-n-1}e^{C_2 m(|\log r|^{-\gamma}+ \delta |\log \dist(x,N)|+r)},$$
 for every $m>m_0$,  $0<r<1/2$ and $0<\delta<a_0/m$. Now choose 
$$r:=e^{-m^{\frac{1}{1+\gamma}}}.$$
 %for some big constant $A_1$ (independent of $m,\alpha,\delta,\varphi$).
  We obtain that 
\begin{align} \label{ine-saukhicoalpha}
|\nabla \psi_{m,\delta}(x)| \le C_3\delta^{-1/2} e^{C_3 m^{\frac{1}{1+\gamma}}+C_3m \delta |\log \dist(x,N)|},
\end{align} 
for every $x\in X$,  $m>m_0$ and $0<\delta<a_0/m$.

Let $\pi: \widehat X \to X$ and  $\widehat N$ be as in Theorem \ref{th-approxthat-desing}. Let $\widehat \psi_{m,\delta}:= \pi^* \psi_{m,\delta}$. Thanks to (\ref{ine-saukhicoalpha}) one gets immediately the following property (which is a stronger version of Theorem \ref{th-approxthat-desing} (iii)):

\begin{equation}\label{eq1thelogcontinuity}
|\nabla \widehat \psi_{m,\delta}(x)| \le C_4\delta^{-1/2} e^{C_4 m^{\frac{1}{1+\gamma}}+C_4m \delta |\log \dist(x, \widehat{N})|},
\end{equation}
for every $x \in \widehat{X}$.

Now arguing exactly as in the proofs of Lemma \ref{le-logcontinuity} (use \eqref{eq1thelogcontinuity} in place of  Theorem \ref{th-approxthat-desing} (iii)) with $\delta:=m^{-2D}, $
we get
\begin{align*}
|\widehat\varphi(x)-\widehat\varphi(y)|  & \lesssim 
m^{-\gamma}-D m^{-2D} \log\dist(x, y)\\
& +  \dist (x, y) m^{D} e^{C_4m^{\frac{1}{1+\gamma}}}   e^{-C_5 m^{-2D+1} \log\dist (x, y)},
\end{align*}
for every $x, y \in \widehat X \backslash \widehat N$ with 
$\left(\dist (x, y)\right)^D \le \min \{ \dist(x,\widehat N), \dist (y,\widehat N)\}.$
  Now letting $$m:=\max\left\{m_0+1, \left(\frac{\gamma |\log\dist (x, y)|}{3 C_4}\right)^{1+\gamma}\right\},$$ we obtain 
$$|\widehat \varphi(x)- \widehat \varphi(y)| \lesssim |\log |x-y||^{-\gamma (1+ \gamma)},$$
for every $x, y \in \widehat X \backslash \widehat N$ with 
$\left(\dist (x, y)\right)^D \le \min \{ \dist(x,\widehat N), \dist (y,\widehat N)\}.$
We note that if $\gamma \to 1$, then $\gamma(1+ \gamma) \to 2$. 
Using again arguments from the proof of Proposition \ref{pro-logalphabecontiu} we infer that Proposition \ref{pro-logalphabecontiu} holds for $\gamma(1+\gamma)$ in place of $\gamma$. Applying now Proposition \ref{pro-chuyendudistNdenlogcon} to $\widehat \varphi$, we see that $\widehat \varphi$ is $\log^{\gamma(1+\gamma)}$-continuous on $\widehat X$. This combined with Lemma \ref{le-successiveblowup} yields that $\varphi$ is $\log^{\gamma'}$-continuous for every $\gamma' \in (0,2)$.   Repeating this procedure gives the desired assertion.
\end{proof}

\section{Log continuity of Monge-Amp\`ere metrics}

In this section we prove  Corollary \ref{cor-metricsemiample}. 
We start with some auxiliary results.  We fix a smooth K\"ahler form $\omega$ on $X$ which induces a distance on $X$. 

\begin{lemma} \label{le-uocluongsoLeleong} Let $v$ be a bounded $\omega$-psh function. Then there exists a constant $C>0$ such that for every $x \in X$ and $\epsilon \in (0,1]$ one has 
$$\lambda(v,x, \epsilon):=\epsilon^{-2n +2}\int_{\B(x,\epsilon)} (\ddc v+ \omega) \wedge \omega^{n-1} \le C/ |\log \epsilon|.$$
\end{lemma}

\proof Fix  $x_0 \in X$. 
Let $\psi$ be a negative $\omega$-psh function on $X$ such that $\psi= c \log |x-x_0|$ on an open neighborhood of $x_0$ in  $X$ for some constant $c>0$,  and $\psi$ is smooth outside $x_0$. For $\epsilon$ small enough, we see that $\psi(x)= c \log |x-x_0|$ on $\B(x,\epsilon)$. Hence $\psi \le  c\log \epsilon$ on $\B(x_0, \epsilon)$. Recall also that 
$$\epsilon^{-2n +2}\int_{\B(x,\epsilon)} (\ddc v+ \omega) \wedge \omega^{n-1} \lesssim \int_{\B(x, \epsilon)} (\ddc v+ \omega) \wedge (\ddc \psi + \omega)^{n-1},$$
see \cite[Page 159]{Demailly_ag}. Hence we get
\begin{align*}
\lambda(v,x, \epsilon) &\lesssim |\log \epsilon|^{-1} \int_{\B(x,\epsilon)} -\psi (\ddc v+ \omega)\wedge (\ddc \psi + \omega)^{n-1}\\
& \le |\log \epsilon|^{-1} \int_X -\psi (\ddc v+ \omega)\wedge (\ddc \psi + \omega)^{n-1}\\
& = |\log \epsilon|^{-1} \int_X -\psi \omega\wedge (\ddc \psi + \omega)^{n-1}\\
& \quad + |\log \epsilon|^{-1} \int_X -v \ddc \psi \wedge \omega\wedge (\ddc \psi + \omega)^{n-1}\\
&\lesssim |\log \epsilon|^{-1} (\|v\|_{L^\infty}+1).
\end{align*}
This finishes the proof.
\endproof

\begin{lemma}\label{le-xapxi} Let $u$ be a $\log^M$-continuous $\theta$-psh function on $X$ such that $u$ is smooth outside $N$. Let $\delta \in (0,1]$ be a constant. Then there exist a constant $C>0$ independent of $\delta$ and a sequence of smooth $(\theta+ \delta \omega)$-psh function $(u_\epsilon)_\epsilon$ so that $u_\epsilon$ converges  uniformly to $u$ and 
$u_\epsilon$ converges to $u$ locally in the $\mathcal{C}^\infty$-topology on $X \backslash N$.
\end{lemma}

\proof This follows essentially from Demailly's regularisation of psh functions (\cite[Section 8]{Demailly-estimationL2}). Let $\exp_z$ be the exponential map at $z \in X$ of $(X,\omega)$. Let $\chi$ be a cut-off function as in \cite[Page 492]{Demailly-estimationL2}. We define
$$u_\epsilon(z):= \frac{1}{C \epsilon^{2n}}\int_{\zeta \in T_zX} u(\exp_z(\zeta)) \chi'(|\zeta|^2/\epsilon^2) d  \lambda(\zeta),$$
where $d \lambda$ denotes the Lebesgue measure on the Hermitian space $T_zX$ and 
$$C:= \int_{\zeta \in T_z X}\chi'(|\zeta|^2) d  \lambda(\zeta).$$ 
One sees immediately that $u_\epsilon$ is $\log^M$-continuous uniformly in $\epsilon$ because $u$ is already $\log^M$-continuous.

By \cite[Proposition 8.5 and Lemma 8.6]{Demailly-estimationL2}, we know that there is a constant $A_1>0$ such that 
$$\ddc u_\epsilon(x) + \theta(x) \ge \big(-A_1 \lambda(u,x, \epsilon)- A_1 \epsilon/ |\log \epsilon|\big) \omega,$$
where 
$$\lambda(u,x, \epsilon):=\epsilon^{-2n +2}\int_{\B(x,\epsilon)} (\ddc u+ \omega) \wedge \omega^{n-1} \le A_2/ |\log \epsilon|$$ 
for some constant $A_2$ independent of $x$ by Lemma \ref{le-uocluongsoLeleong}. Hence we infer that 
 $$\ddc u_\epsilon + \omega\ge -A_3 |\log  \epsilon|^{-1} \omega$$
for some constant $A_3>0$ independent of $\epsilon$. This finishes the proof.  
\endproof

\begin{lemma}\label{le-hamkhoangcach} Let $\delta \in (0,1]$, $M>1$ and $C_0>0$ be constants.
	 Let $u$ be
	a smooth $(\theta+\delta \omega)$-psh functions such that 
$$|u(x)- u(y)| \le C_0 |\log \dist (x,y)|^{-2M}$$
 for every $x,y \in X$.    
Denote by $\tilde{d}$ the distance induced by $\ddc u+ \theta+ \delta \omega$. Then
$$\tilde{d}(x, y) \le C |\log \dist(x,y)|^{-M+1}$$
for every $x,y \in X$, where $C>0$ is a constant independent of $u$ and $\delta$.  
\end{lemma}

\begin{proof}
Let $\Omega(r)$ be the modulus of continuity of $u$. By hypothesis, one has
\begin{align}\label{ine-Moduluscontinuity}
\Omega(r) \le C_0 |\log r|^{-2 M}
\end{align}
for every $0 < r <1$.  We cover $X$ by finitely many local charts (which are relatively compact in bigger local charts) and since the K\"ahler form $\omega$ is equivalent to the standard K\"ahler form on $\C^n$ in these local charts, we can assume that $\omega$ is equal to the standard form on $\C^n$  on these local charts.

Let $\B(x,r)$ denotes the ball of radius $r$ with center at $x \in \C^n$.
Fix $x^* \in X$ and a local chart $U$ around $x^*$ biholomorphic to $\B(0,2)$ such that $x^* =0$ in these local coordinates. Define $d(x):= \tilde{d}(x,0)$. Recall that  $\tilde{d}$ is the Riemannian metric induced by $\ddc u + \theta+ \delta\omega$.   For $x \in \B(0,1)$,   let
$$d_{r}(x):=\vol(\B(x, r))^{-1}\int_{x' \in \B(x,r)} d(x')\, \omega^n.$$
Arguing as in the proof of \cite[Lemma 5]{Guo-Phong-Tong-Wang} (see also  \cite{YangLi} or \cite{BinGuo}) and using (\ref{ine-Moduluscontinuity}), for every  $x_0 \in \B(0,1)$, one obtains
$$\int_{\B(x_0, r)} |\nabla d|^2_\omega  \omega^n \le C_1 r^{2n}+ C_1 \int_{\B(x_0, 3r/2)} |u(x)- u(x_0)| \omega^n \le  C_2 r^{2n-2} |\log r|^{-2 M}.$$
Therefore, by Poincar\'e inequality, we infer
$$r^{-2n} \int_{\B(x_0,r)}|d(x)- d_{r}(x_0)|^2 \omega^n \le C_3 |\log r|^{-2 M},$$
where $C_3>0$ is a uniform constant independent of $x_0, \delta$ and $r$.  This combined with H\"older inequality gives
\begin{align} \label{ine-tichphandFOmega}
    r^{-2n} \int_{\B_\omega(x_0,r)}|d(x)- d_{r}(x_0)| \omega^n \le C_3 |\log r|^{-M}.
\end{align}
We now use some arguments similar to the proof of Campanato's lemma. We follow the presentation in \cite[Chapter 3]{Han-Lin-book}. Assume
 $0 <r_1<r_2<1$ and $x_1, x_2\in\B(0, 1)$ with $\B(x_1, r_1)\subset \B(x_2, r_2)$. Observe that 
$$|d_{r_1}(x_0)- d_{r_2}(x_0)| \le |d_{r_1}(x_0)- d(x)|+ |d_{r_2}(x_0)- d(x)|$$
for every $x \in \B(x_1, r_1)\subset \B(x_2, r_2)$. It follows that
\begin{multline*}
|d_{r_1}(x_1)- d_{r_2}(x_2)| \le \\
\vol(\B_\omega(x_1, r_1))^{-1} \bigg(\int_{\B_\omega(x_1, r_1)}|d_{r_1}(x_1)- d(x)| \omega^n + \int_{\B_\omega(x_2, r_2)}|d_{r_2}(x_2)- d(x)| \omega^n \bigg) 
\end{multline*}
By (\ref{ine-tichphandFOmega}), it follows that
\begin{equation}\label{eq_campanato}
|d_{r_1}(x_1)- d_{r_2}(x_2)|\lesssim r_1^{-2n} \big(r_1^{2n} |\log r_1|^{-M} +  r_2^{2n} |\log r_2|^{-M}\big).
\end{equation}
Applying the last inequality to $r_1= r/2^{k+1}, r_2= r/ 2^k$ and $x_1=x_2=x_0$ yields
$$|d_{2^{-k} r}(x_0)- d_{2^{-k-1}r}(x_0)| \le  C_4\big(|\log r|+ (k+1) \log 2\big)^{-M}\leq 
\int_{k}^{k+1}\frac{C_4 dt}{(|\log r|+ t \log 2)^{-M}},$$
for some uniform constant $C_4>0$ independent of $x_0, r,k, \epsilon,\delta$.
Summing over $k=0, 1, 2, \ldots$ yields
$$ \sum_{k=0}^\infty |d_{2^{-k} r}(x_0)- d_{2^{-k-1}r}(x_0)| \le \int_{0}^{\infty}\frac{C_4 dt}{(|\log r|+ t \log 2)^{-M}}\leq C_5 |\log r|^{-M+1}.$$
 Since $d_{r}$ converges uniformly to $d$ on $\B (0, 1)$ as $r\searrow 0$, it follows that
\begin{align}\label{ine-dFx0truditrungbinh}
|d(x_0)- d_{r}(x_0)| \le C_5 |\log r|^{-M+1}
\end{align}
for every $x_0 \in \B(x^*, 1)$ (recall $x^*= 0$), and $0<r< 1$. Let $x \in \B(x^*, 1/8)$ and $r:= \dist(x,x^*) \le 1/8$. Using
 \eqref{ine-dFx0truditrungbinh} and then  applying \eqref{eq_campanato} for $x_1=x$, $x_2=x^*$ and $r_2=3 r_1=3 r$, we get
\begin{align*}
d(x) &= d(x)- d(x^*) \\
&\le |d(x)- d_{r}(x)|+ |d(x^*)- d_{3r}(x^*)|+ |d_{r}(x)- d_{3r}(x^*)| \\
&\lesssim |\log r|^{-M+1} +|d_{r}(x)- d_{3r}(x^*)|\\
&\lesssim |\log r|^{-M+1} +|\log r|^{-M}\\
&\lesssim |\log r|^{-M+1}.
\end{align*}
This finishes the proof.
\end{proof}

\begin{proof}[Proof of Corollary \ref{cor-metricsemiample}]  Write $\omega_F= \ddc u+ \theta$, where $u$ is $\log^M$-continuous $\theta$-psh function for every constant $M>1$ by Theorem \ref{the-log-continuity}.
Fix $\delta \in (0,1]$.  Let $u_\epsilon$ be as in Lemma \ref{le-xapxi} for $u$. Since $\ddc u_\epsilon +\theta+ \delta \omega$ converges to $\ddc u+ \theta+ \delta \omega \ge  \ddc u+ \theta$ locally in the $\mathcal{C}^\infty$-topology in $X \backslash N$, one sees that the desired assertion follows from   Lemma \ref{le-hamkhoangcach} by letting $\epsilon \to 0$. 
 \end{proof}

\bibliography{biblio_family_MA,biblio_Viet_papers,bib-kahlerRicci-flow}

\bibliographystyle{siam}

\bigskip

\noindent
\Addresses
\end{document}